\documentclass{article}

\usepackage{a4}
\usepackage{psfig}
\usepackage{fancyhdr}
\usepackage{latexsym}
\usepackage{amssymb}
\usepackage{mathtools}

\usepackage{abstract}

\usepackage[subnum]{cases}

\usepackage{amstext}
\usepackage{amsfonts}
\usepackage{mathrsfs}
\usepackage{multirow}
\usepackage{multicol}
\usepackage{stackrel}

\usepackage{tikz}

\usepackage[square]{natbib}
\setcitestyle{numbers}

\usepackage[dvips]{epsfig}
\usepackage[latin1]{inputenc}

\usepackage{csquotes}

\usepackage{float}
\usepackage{verbatim}
\usepackage{longtable}
\usepackage{tabularx}       % Fuer Tabellen laenger als eine Seite
\usepackage{multicol}
\usepackage{color}
\usepackage{wrapfig}
\usepackage{floatflt}
\usepackage{afterpage}

\usepackage{textcomp}
\usepackage[gen]{eurosym}

\usepackage{amsmath}
\usepackage{amssymb}
\usepackage{amstext}
\usepackage{mathrsfs}
\usepackage{amsfonts}
\usepackage{bbm}

\usepackage{graphicx}
\usepackage{epsfig}
\usepackage{subfigure}

\usepackage{nomencl}
\usepackage{epigraph}
\makenomenclature

\usepackage{url}
\usepackage{appendix}

\usepackage{amsthm}
\usepackage{yfonts}

\newcommand{\F}{\mathbb{F}}

\newcommand{\R}{\mathbb{R}}

\newcommand{\N}{\mathbb{N}}
\newcommand{\PP}{\mathbb{P}}
\newcommand{\EE}{\mathbb{E}}
\newcommand{\dd}{\mathrm{d}}

\DeclareMathOperator*{\argmin}{arg\,min}
\DeclareMathOperator*{\arginf}{arg\,inf}
\newcommand\restr[2]{{% we make the whole thing an ordinary symbol
  \left.\kern-\nulldelimiterspace % automatically resize the bar with \right
  #1 % the function
  \vphantom{\big|} % pretend it's a little taller at normal size
  \right|_{#2} % this is the delimiter
  }}

\newcommand{\indicator}[1]{\mathbbm{1}_{\left\{ {#1} \right\} }}
\newcommand{\E}[1]{\mathbb{E}\left[ {#1} \right] }
\newcommand{\Prob}[1]{\PP \left[ {#1} \right] }

\newcommand{\pp}[1]{\left( {#1} \right)^+ }

\theoremstyle{plain}
\newtheorem{Theorem}{Theorem}[section]
\newtheorem{Proposition}[Theorem]{Proposition}%[section]
\newtheorem{Lemma}[Theorem]{Lemma}%[section]
\newtheorem{Definition}[Theorem]{Definition}%[section]
\newtheorem{Assumption}{Assumption}%[section]
\newtheorem{Example}[Theorem]{Example}%[section]

\theoremstyle{remark}
\newtheorem{Remark}[Theorem]{Remark}%[section]

\numberwithin{equation}{section}
\numberwithin{figure}{section}

\usepackage[
bookmarks,
bookmarksopen=true,
pdftitle={},
pdfauthor={Jan Obloj, Peter Spoida},
pdfcreator={Jan Obloj, Peter Spoida},
pdfsubject={},
pdfkeywords={},
colorlinks=false,
anchorcolor=black,% Ankertext
filecolor=magenta, % Verknpfungen, die lokale Dateien ffnen
menucolor=red, % Acrobat-Menpunkte
linkcolor=black, % einfache interne Verknpfungen
plainpages=false,% zur korrekten Erstellung der Bookmarks
pdfpagelabels,% zur korrekten Erstellung der Bookmarks
hypertexnames=false,% zur korrekten Erstellung der Bookmarks
linktocpage % Seitenzahlen anstatt Text im Inhaltsverzeichnis verlinken
]{hyperref}

\title{An Iterated Az\'{e}ma-Yor Type Embedding for Finitely Many Marginals}
\author{Jan Ob\l{}\'{o}j \thanks{Jan Ob\l\'oj is thankful to the Oxford-Man Institute of Quantitative Finance and St John's College in Oxford for their support.  \newline
\href{mailto:jan.obloj@maths.ox.ac.uk; web}{jan.obloj@maths.ox.ac.uk}; \url{http://www.maths.ox.ac.uk/people/profiles/jan.obloj}} \hspace{3mm} and \hspace{3mm} Peter Spoida\thanks{Peter Spoida gratefully acknowledges scholarships from the Oxford-Man Institute of Quantitative Finance and the DAAD. \newline
\href{mailto:peter.spoida@maths.ox.ac.uk}{peter.spoida@maths.ox.ac.uk}; \url{http://www.maths.ox.ac.uk/people/profiles/peter.spoida}} \medskip\\
University of Oxford\thanks{Mathematical Institute, ROQ, Woodstock Rd, Oxford OX2 6GG, UK}
}

\date{\today}

\begin{document}

\maketitle
%\tableofcontents

\begin{abstract}
 
We solve the $n$-marginal Skorokhod embedding problem for a continuous local martingale and a sequence of probability measures $\mu_1,\dots,\mu_n$ which are in convex order and satisfy an additional technical assumption. Our construction is explicit and is a multiple marginal generalisation of the \citet{AzemaYor1} solution. In particular, we recover the stopping boundaries obtained by \citet{Brown98themaximum} and \citet{MadanYor}. 
Our technical assumption is necessary for the explicit embedding, as demonstrated with a counterexample. We discuss extensions to the general case giving details when $n=3$.

In our analysis we compute the law of the maximum at each of the $n$ stopping times.
This is used in \citet{Touzi_maxmax} to show that the construction maximises the distribution of the maximum among all solutions to the $n$-marginal Skorokhod embedding problem. The result has direct implications for robust pricing and hedging of Lookback options.% given prices of call options with multiple intermediate maturities.
\end{abstract}
%\textbf{Keywords}: Skorokhod embedding problem, Az\'ema-Yor embedding, maximum process,
%\smallskip\\

\newpage

\section{Introduction}

We consider here an $n$-marginal Skorokhod embedding problem (SEP). We construct an explicit solution which has desirable optimal properties. % in that it maximises the distribution of the maximum among all the embeddings.
The classical (one-marginal) SEP consists in finding a stopping time $\tau$ such that a given stochastic process $(X_t)$ stopped at $\tau$ has a given distribution $\mu$. For the solution to be useful (and non-trivial) one further requires $\tau$ to be \emph{minimal} (cf.\ \citet[Sec.~8]{Obloj:04b}). When $X$ is a continuous local martingale and $\mu$ is centred in $X_0$, this is equivalent to $(X_{t\land \tau}:t \geq 0)$ being a uniformly integrable martingale.
The problem, dating back to the original work in \citet{Skorokhod:65}, has been an active field of research for nearly 50 years. New solutions often either considered new classes of processes $X$ or focused on finding stopping times $\tau$ with additional optimal properties. This paper contributes to the latter category. We are motivated, as was the case for several earlier works in the field, by questions arising in mathematical finance which we highlight below.
\medskip\\
\textbf{The problem and main results}. To describe the problem we consider, take a standard Brownian motion $B$ and a sequence of probability measures $\mu_1,\ldots,\mu_n$. A solution to the $n$-marginal SEP is a sequence of stopping times $\tau_1 \leq \dots \leq \tau_n$ such that $B_{\tau_i} \sim \mu_i$, $1\leq i\leq n$, and $\left( B_{t \wedge \tau_n} \right)_{t \geq 0}$ is a uniformly integrable martingale. 
It follows from Jensen's inequality that a solution may exist only if all $\mu_i$ are centred and the sequence is in convex order. And then it is easy to see how to solve the problem: it suffices to iterate a solution to the classical case $n=1$ developed for a non-trivial initial distribution of $B_0$, of which several exist. 

In contrast, the question of optimality is much more involved.
In general there is no guarantee that a simple iteration of optimal embeddings would be globally optimal. Indeed, this is usually not the case. Consider the embedding of \citet{AzemaYor1} which consists of a first exit time for the joint process $(B_t, \bar{B}_t)_{t \geq 0}$, where $\bar{B}_t=\sup_{s\leq t} B_s$. More precisely, their solution $\tau^{\mathrm{AY}} = \inf\left\{ t \geq 0: B_t \leq \xi_{\mu}(\bar{B}_t) \right\}$ leads to a functional relation $B_{\tau^{\mathrm{AY}}} = \xi_{\mu}(\bar{B}_{\tau^{\mathrm{AY}}})$. This then translates into the optimal property that the distribution of $\bar{B}_{\tau^{\mathrm{AY}}}$ is maximized in stochastic order amongst all solutions to SEP for $\mu$, i.e. for all $y$,
\begin{align*}
\PP\left[ \bar{B}_{\tau^{\mathrm{AY}}} \geq y \right]  = \sup\left\{\PP\left[ \bar{B}_{\rho} \geq y \right]:\rho \textrm{ s.t. }B_{\rho}\sim \mu, (B_{t\land \rho}) \textrm{ is UI }\right\}.
\end{align*} 
It is not hard to generalise the Az\'ema-Yor embedding to a non-tirivial starting law, see \citet[Sec.~5]{Obloj:04b}. Consequently we can find $\eta_i$ such that $\tau_{i} = \inf\left\{ t \geq \tau_{i-1}: B_t \leq \eta_i(\sup_{\tau_{i-1}\leq s\leq t}B_s) \right\}$ solve the $n$-marginal SEP. However this construction will maximise stochastically the distributions of $\sup_{\tau_{i-1}\leq t\leq \tau_i} B_t$, for each $1\leq i\leq n$, but not of the global maximum $\bar{B}_{\tau_n}$. The latter is achieved with a new solution which we develop here.

Our construction involves an interplay between all $n$-marginals and hence is not an iteration of a one-marginal solution. However it preserves the spirit of the Az\'{e}ma-Yor embedding in the following sense. Each $\tau_{i}$ is still a first exit for $(B_{t}, \bar{B}_{t})_{t \geq \tau_{i-1}}$ which is designed in such a way as to obtain a \enquote{strong relation} between $B_{\tau_{i}}$ and $\bar{B}_{\tau_{i}}$, ideally a functional relation. 
Under our technical assumption about the measures $\mu_1, \dots,\mu_n$, Assumption \ref{ass:unicity_minimizers}, we describe this relation in detail in Lemma \ref{lem:Key_Properties_of_the_Stopping_Rule}.
% which is key for the pathwise proof of the optimal properties of our embedding in Section 4 in \citet{Touzi_maxmax}.

For $n=2$ we recover the results of \citet{MAFI:MAFI116}. We also recover the trivial case $\tau_i=\tau^{AY}_{\mu_i}$ which happens when $\xi_{\mu_i} \leq \xi_{\mu_{i+1}}$, we refer to \citet{MadanYor} who in particular then investigate properties of the arising time-changed process. 
However, as a counterexample shows, our construction does not work for all laws $\mu_1, \dots,\mu_n$ which are in convex order. Assumption \ref{ass:unicity_minimizers} fails when a special interdependence between the marginals is present and
 the analysis then becomes more technical and the resulting quantities are, in a way, less explicit. We only detail the appropriate arguments for the case $n=3$.

We stress that the problem considered in this paper is significantly more complex that the special case $n=1$. For $n=1$ several solutions to SEP exist with different optimal properties. For $n=2$ only one such construction, the generalisation of the Az\'ema--Yor embedding obtained by \citet{Brown98themaximum}, seems to be known. To the best of our knowledge, the solution we present here is the first one to deal with the general $n$-marginal SEP. 
\medskip\\
\textbf{Motivation and applications}. 
Our results have direct implications for, and were motivated by, robust pricing and hedging of lookback options. In mathematical finance, one models the price process $S$ as a martingale and specifying prices of call options at maturity $T$ is equivalent to fixing the distribution $\mu$ of $S_T$. Understanding no-arbitrage price bounds for a functional $O$, which time-changes appropriately, is then equivalent to finding the range of $\E{O(B)_\tau}$ among all solutions to the Skorokhod embedding problem for $\mu$. This link between SEP and robust pricing and hedging was pioneered by \citet{Hobson:98b} who considered Lookback options. Barrier options were subsequently dealt with by \citet{MAFI:MAFI116}. More recently, \citet{CoxObloj:08, CoxObloj:09} considered the case of double touch/no-touch barrier options, \citet{HobsonNeuberger:12} looked at forward starting straddles and analysis for variance options was undertaken by \citet{CoxWang:11}. We refer to \citet{Hobson2010survey} and \citet{OblojEQF:10} for an exposition of the main ideas and more references. However, all the previous works considered essentially the case of call options with one maturity, i.e.\ a one-marginal SEP, while in practice prices for many intermediate maturities may also be available. This motivated our investigation.

We started our quest for a general $n$-marginal optimal embedding by computing the value function $\sup \E{\phi(\sup_{t\leq \tau_n} B_t)}$ among all solutions to the $n$-marginal SEP. This was achieved using stochastic control methods, developed first for $n=1$ by \citet{Touzi11}, and is reported in a companion paper by \citet{Touzi_maxmax}. Knowing the value function we could start guessing the form of the optimiser and this led to the present paper. Consequently the optimal properties of our embedding, namely that it indeed achieves the value function in question, are shown by \citet{Touzi_maxmax}. In fact we give two proofs in that paper, one via stochastic control methods and another one by constructing appropriate pathwise inequalities and exploiting the key Lemma \ref{lem:Key_Properties_of_the_Stopping_Rule} below, cf. \citet[Section 4]{Touzi_maxmax}.
\medskip\\
\textbf{Organisation of the paper}. 
The remainder of the paper is organized as follows.
In Section \ref{sec:Main Assumption and Definitions} we explain the main quantities for the embedding and state the main result. We also present the restriction on the measures $\mu_1,\dots,\mu_n$ which we require for our construction to work (Assumption \ref{ass:unicity_minimizers}). 
In Section \ref{sec:iAY} we prove the main result and Section \ref{sec:Extensions} provides a discussion of extensions together with comments on Assumption \ref{ass:unicity_minimizers}.
The proof of an important but technical lemma is relegated to the Appendix.

\section{Main Result}
\label{sec:Main Assumption and Definitions}

Let $\left( \Omega, \mathcal{F}, \F, \PP \right)$, where $\F=(\mathcal{F}_t)$, be a filtered probability space satisfying the usual hypothesis and $B$ a continuous $\F$--local martingle, $B_0=0$, 
$\langle B \rangle_{\infty} = \infty$ a.s.\ and $B$ has no intevals of constancy a.s.
We denote $\bar{B}_t:=\sup_{s \leq t}B_t$. We are primarily interested in the case when $B$ is a standard Brownian motion and it is convenient to keep this example in mind, hence the notation. We allow for more generality as this introduces no changes to the statements or the proofs.

\subsection{Definitions}
 
The following definition will be crucial in the remainder of the article. We define the stopping boundaries $\xi_1,\dots,\xi_n$ for our iterated Az\'{e}ma-Yor type embedding together with quantities $K_1,\dots,K_n$ which will be later linked to the law of the maximum at subsequent stopping times.

\begin{Definition}
\label{def:indices}
Fix $n \in \N$. For convenience we set  
\begin{equation}
\begin{split}
&c_0\equiv 0, \quad K_0 \equiv 0, \quad \xi_{0} \equiv -\infty. % \quad \xi_{n+1}(y) \quad \text{for $y \geq 0$}.
\end{split}
\label{eq:conventions}
\end{equation}

For $\zeta \in \R$ and $i=1,\dots,n$ we write 
\begin{align}
c_i(\zeta) := \int_{\R}{\left(x-\zeta\right)^+\mu_i(\dd x)}.
\label{eq:definition_call_price}
\end{align}
 
Let $y \geq 0$ and assume that for $i=1,\dots,n-1$ the quantities $\xi_i, K_i, \imath_i$ and $ \jmath_i$ are already defined. Then we define
\begin{equation}
\begin{split}
&\imath_n(\cdot;y): (-\infty,y] \to \left\{ 0,1,\dots,n-1 \right\}, \\ 
&\zeta \mapsto \imath_n(\zeta;y) := \max \left\{k \in \{ 0,1,\dots,n-1 \}: \xi_{k}(y) < \zeta \right\},
\end{split}
\label{eq:definition_imath}
\end{equation} 
and
\begin{align}
\xi_n(y) := \sup \left\{ \arginf_{\zeta < y} \left(  \frac{c_n(\zeta)}{y-\zeta} - \left[ \frac{c_{\imath_n(\zeta;y)}(\zeta)}{y-\zeta} - K_{\imath_n(\zeta;y)}(y) \right]   \right) \right\}.
\label{eq:definition_xi_n}
\end{align}

With
\begin{align}
\jmath_n(y) := \imath_n(\xi_n(y);y)
\end{align}
we set
\begin{align}
K_n(y):=\frac{1}{y-\xi_n(y)} \left\{ \vphantom{\frac{1}{y-\xi_n(y)}} c_n(\xi_n(y)) - \left[ c_{\jmath_n(y)}(\xi_n(y)) - (y-\xi_n(y)) K_{\jmath_n(y)}(y) \right] \right\}.
\label{eq:definition_Kn}
\end{align}
\end{Definition}

\begin{Definition}[Embedding]
\label{def:iterated_AY}
Set $\tau_0 \equiv 0$ and for $i=1,\dots,n$ define  
\begin{numcases}{\tau_i:=}
\inf\left\{ t \geq \tau_{i-1}:B_t \leq \xi_{i}(\bar{B}_t) \right\} & \text{if $B_{\tau_{i-1}}> \xi_{i}(\bar{B}_{\tau_{i-1}}),$} \label{eq:embedding1} \\
\tau_{i-1} & \text{else.}
\label{eq:embedding3}
\end{numcases} 
\end{Definition}

\begin{figure}[h!]
\label{fig:boundaries_illustration}
	\centering
		\includegraphics[scale=0.85]{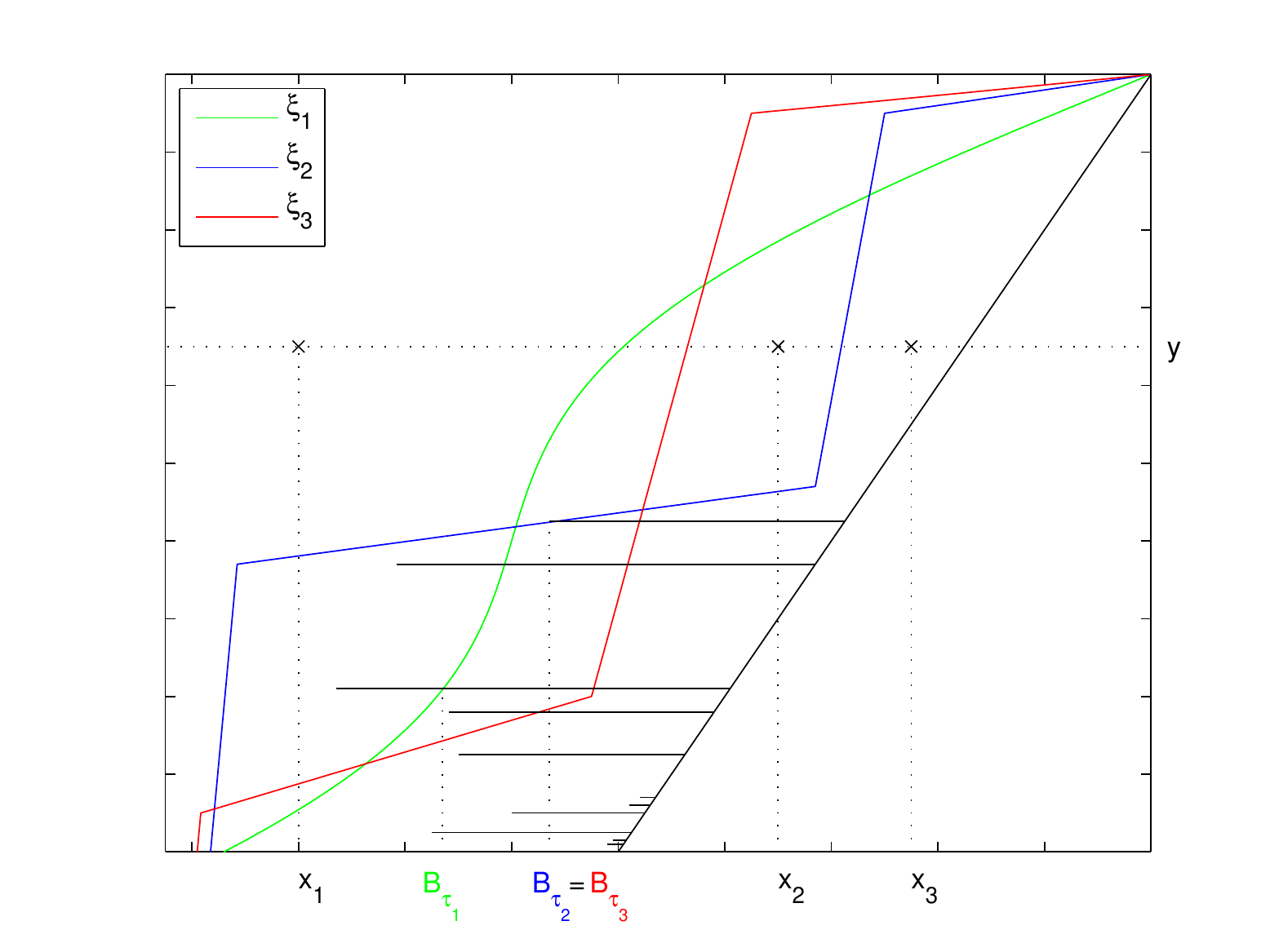}
	\caption{We illustrate possible stopping boundaries $\xi_1,\xi_2,\xi_3$. The horizontal lines represent a sample path of the process $\left( B_t, \bar{B}_t\right)$ where the $x$-axis is the value of $B$ and the $y$-axis the value of $\bar{B}$. 
Each horizontal segment is an excursion of $B$ away from its maximum $\bar{B}$.
According to the definition of the embedding, the first stopping time $\tau_1$ is found when the process first hits $\xi_1$. 
Since $\xi_1(\bar{B}_{\tau_1}) > \xi_2(\bar{B}_{\tau_1})$ the process continues and targets $\xi_2$. The stopping time $\tau_2$ is found when the process first hits $\xi_2$. Since $\xi_2(\bar{B}_{\tau_2}) \leq \xi_3(\bar{B}_{\tau_2})$ we get $\tau_3 = \tau_2$. For the $y$ we fixed we have $\imath_3(x_1;y) = 0,\imath_3(x_2,y) = 1,\imath_3(x_3;y) = 2$.}
\end{figure}

\begin{Remark}[Properties of $\imath_n$]
Recalling the definition of $\imath_n$, cf. \eqref{eq:definition_imath}, we observe for later use that for $y \geq 0$
\begin{align}
\text{$\imath_n(\cdot;y)$ is left-continuous and has at most $n-1$ jumps}
\label{eq:left_continuity_imath}
\end{align}
and for $x \in \R$
\begin{align}
\text{$\imath_n(x;\cdot)$ is right-continuous and has at most $n-1$ jumps.}
\label{eq:right_continuity_imath}
\end{align}
\end{Remark}

Figure \ref{fig:boundaries_illustration} illustrates a set of possible stopping boundaries $\xi_1,\xi_2,\xi_3$ in the case of $n=3$. If Assumption \ref{ass:unicity_minimizers} is in place, see Section \ref{subsec:restriction}, we will show that the stopping boundaries are continuous (except possibly for $i=1$) and non-decreasing, cf. Section \ref{subsec:Basic Properties of Stopping Boundaries}.

The $n^{\mathrm{th}}$ stopping boundary $\xi_n$ is obtained from an optimization problem which features $\xi_1,\dots,\xi_{n-1}$ and $K_1,\dots,K_{n-1}$. $K_n(y)$ is the value of the objective function at the optimal value $\xi_n(y)$. 
Note that all previously defined stopping boundaries $\xi_1,\dots,\xi_{n-1}$ and the quantities $K_1,\dots,K_{n-1}$ remain unchanged. 

%Now we explain the relation of the (one-marginal) classical Az\'{e}ma-Yor embedding to our embedding.
Denote the right and left endpoints of the support of the measure $\mu_i$ by 
\begin{align}
r_{\mu_i} := \inf \left\{ x:\mu_i \left((x,\infty)\right)=0 \right\}, \qquad
l_{\mu_i} := \sup \left\{ x:\mu_i \left([x,\infty)\right)=1 \right\}
\label{eq:right_endpoint_support},
\end{align}
respectively, 
and the barycentre function of $\mu_i$ by
\begin{align}
b_i(x):= \frac{\int_{[x,\infty)}{u \dd \mu_i(u)}}{\mu_i([x,\infty))} \indicator{x<r_{\mu_i}} + x \indicator{x \geq r_{\mu_i}},
\label{eq:def_barycenter}
\end{align} 

As shown by \citet{MAFI:MAFI116}, the right-continuous inverse of $b_i$, denoted by $b_i^{-1}$, can be represented as
\begin{align}
b_i^{-1}(y) =  \sup \left\{ \arginf_{\zeta < y}   \frac{c_i(\zeta)}{y-\zeta} \right\}.
\label{eq:AY_minimization}
\end{align}

It is clear and has been studied in more detail by \citet{MadanYor} that if the sequence of barycentre functions is increasing in $i$, then the intermediate law constraints do not have an impact on the corresponding iterated Az\'{e}ma-Yor embedding.
However, in general the barycentre functions will not be increasing in $i$, cf. \citet{Brown98themaximum}, and hence will affect the embedding.
We think of $\jmath_n(y)$ as the index of the last law $\mu_i, i<n$, which represents, locally at level of maximum $y$, a \textit{binding constraint} for the embedding.
As compared to the optimization from which $b^{-1}_n$ is obtained, cf. \eqref{eq:AY_minimization}, the optimization from which $\xi_n$ is obtained, cf. \eqref{eq:definition_xi_n}, has a penalty term.

\subsection{Restrictions on Measures} 
\label{subsec:restriction}

Throughout the article we will denote the left- and right-limit of a function $f$ at $x$ (if it exists) by $f(x-)$ and $f(x+)$, respectively.

Recalling the conventions in \eqref{eq:conventions}, we define inductively for $n \in \N$ and $y \geq 0$ the mappings
\begin{equation}
\begin{split}
& c^n(\cdot,y):(-\infty, y] \to \R \cup \left\{ \infty \right\},  \\
& x \mapsto c^n(x,y):=c_n(x) - \left[ c_{\imath_n(x;y)}(x) - (y-x) K_{\imath_n(x;y)}(y) \right].
\label{eq:def_c_n}
\end{split}
\end{equation}

It follows that the minimization problem in \eqref{eq:definition_xi_n} is equivalent to the following minimization problem, 
\begin{align}
\xi_n(y) \in \argmin_{\zeta \leq y} \frac{c^n(\zeta,y)}{y-\zeta},
\label{eq:definition_xi_n_tangent_interpretation}
\end{align}
where we observe that
\begin{align}
\restr{\frac{c^n(\zeta,y)}{y-\zeta}}{\zeta=y}:=&\lim_{\zeta \uparrow y} \frac{c^n(\zeta,y)}{y-\zeta} 
\label{eq:extension_c_n}  \\
=&\begin{cases} 
-c'_n(y-)+c'_{\imath_n(y;y)}(y-) + K_{\imath_n(y;y)}(y) & \text{if $c_n(y)=c_{\imath_n(y,y)}(y)$,} %\label{eq:extension_c_n_a} 
\\
+\infty & \text{else.}
%\label{eq:extension_c_n_b}
\end{cases} 
\nonumber
\end{align}
Now we want to argue existence in \eqref{eq:definition_xi_n_tangent_interpretation}.
In the case $y>0$ this can be deduced iteratively from the -- a priori -- piecewise continuity of $c^n(\cdot,y)$ and the fact that $c^n \geq 0$ together with the property that $\zeta \mapsto \frac{c^n(\zeta,y)}{y-\zeta} = \frac{c_n(\zeta)}{y-\zeta}$ for $\zeta$ sufficiently small, which is a non-increasing function. This is because, inductively, $\xi_1(y),\dots,\xi_{n-1}(y)$ are finite and fixed and hence $\imath_n(\zeta;y)=0$ for $\zeta < \min_{i<n}\xi_i(y)$.

For $y \geq 0$, we extend 
\begin{align}
\restr{\frac{c^n(\zeta,y)}{y-\zeta}}{\zeta=l_{\mu_n}} := 
\begin{cases} 
\frac{-l_{\mu_n}}{y - l_{\mu_n}}, & \text{if $l_{\mu_n}> - \infty$,} 
\\
1 & \text{else.}
%\label{eq:extension_c_n_b}
\end{cases}
%\nonumber
\end{align}
%For $y=0$ this is consistent since $\frac{c_i(\zeta)}{0-\zeta} \to 1$ as $\zeta \to l_{\mu_i}$ and, iteratively, $K_i(0) = 1$. Then it follows that $\xi_n(0) = l_{\mu_n}$.

For later use observe  
\begin{align}
\min_{i \leq n} b_i^{-1}(y) \leq \xi_n(y) \leq y
\label{eq:lower_bound_for_xi_n}
\end{align}
which follows from the definition of $\xi_n$, cf. \eqref{eq:definition_xi_n}, and
where $b_i^{-1}$ denotes the right-continuous inverse of the barycentre function $b_i$, cf. \eqref{eq:AY_minimization}.

\begin{Assumption}[Restriction on Measures]
\label{ass:unicity_minimizers} 
Recall definitions in \eqref{eq:conventions}--\eqref{eq:definition_call_price}, \eqref{eq:def_c_n} and \eqref{eq:right_endpoint_support}.
We impose the following restrictions on the measures $\mu_1,\dots,\mu_n$:
\begin{itemize}
	\item[(i)] $\int |x|\mu_i(\dd x)<\infty$ with $\int x \mu_i(\dd x)=0$ and $c_{i-1}\leq c_i$ for all $1\leq i\leq n$,
	\item[(ii)] for all $2 \leq i \leq n$ and all $0 < y < r_{\mu_i}$ the mapping 
		\begin{align}
		\hspace{-7mm}	[l_{\mu_i},y] \to \R \cup \{ +\infty \}, \quad  \zeta \mapsto \frac{c^i(\zeta,y)}			{y-\zeta} \quad \text{ has a unique minimizer $\zeta^{\star}$}
		\label{eq:suff_cond_unicity_minimzers}
		\end{align} 
		and
		\begin{align}
		\hspace{-7mm} c_i(y)>c_{\imath_i(y;y)}(y) \qquad \text{whenever $\zeta^{\star}<y$}.
		\label{eq:assumption_strictly_ordered_calls}
		\end{align}
\end{itemize} 
\end{Assumption}

\begin{Remark}[Assumption \ref{ass:unicity_minimizers}]
The condition that the call prices are non-decreasing in maturity 
\begin{align}
c_i \leq c_{i+1}, \qquad i=1,\dots,n-1,
\label{eq:ordered_calls}
\end{align}
can be rephrased by saying that $\mu_1, \dots, \mu_n$ are non-decreasing in the convex order. Condition $(i)$ in Assumption \ref{ass:unicity_minimizers} is the necessary and sufficient condition for a uniformly integrable martingale with these marginals to exist, as shown by e.g.\ \citet[Theorem 2]{Strassen65} or \citet[Chapter XI]{meyer1966probability}.

Condition $(ii)$ in Assumption \ref{ass:unicity_minimizers} will be discussed further in section \ref{sec:Extensions}.

Note that if \eqref{eq:ordered_calls} holds with strict inequality then \eqref{eq:assumption_strictly_ordered_calls} is automatically satisfied.
\end{Remark}

\begin{Remark}[Discontinuity of $\xi_1$]
\label{rem:xi_1_discontinuous}
Note that Assumption \ref{ass:unicity_minimizers}(ii) does not require that the mapping
\begin{align}
\zeta \mapsto \frac{c^1(\zeta,y)}{y-\zeta}  = \frac{c_1(\zeta,y)}{y-\zeta}
\label{eq:xi_1_discontinuous}
\end{align}
has a unique minimizer. It may happen that there is an interval of minimizers and then $\xi_1$ is discontinuous at such $y$.
\end{Remark}

%\begin{Remark}[Existence of Minimizers]
%Existence of a minimizer of \eqref{eq:suff_cond_unicity_minimzers} follows from the continuity of $c^i(\cdot,\cdot)$ which we show in Lemma \ref{lem:continuity_c_power_n}.
%\end{Remark}  

%\citet{henrylabordere:hal-00790001}

\subsection{The Main Result}

Our main result shows how to iteratively define an embedding of $(\mu_1, \dots, \mu_n)$ in the spirit of \citet{AzemaYor1} and \citet{Brown98themaximum} if Assumption \ref{ass:unicity_minimizers} is in place. 

\begin{Theorem}[Main Result]
\label{thm:main_result}
Let $n \in \N$ and assume that the measures $\mu_1,\ldots,\mu_n$ satisfy Assumption \ref{ass:unicity_minimizers} from Section \ref{subsec:restriction}. 
Recall Definitions \ref{def:indices} and \ref{def:iterated_AY}. 
Then $\tau_i < \infty$, $B_{\tau_i} \sim \mu_i$ for all $i=1,\dots,n$ and $\left( B_{\tau_n \wedge t } \right)_{t \geq 0}$ is a uniformly integrable martingale.

In addition, we have for $y \geq 0$ and $i=1,\dots,n,$
\begin{align}
\PP\left[ \bar{B}_{\tau_i} \geq y \right] = K_i(y)
\end{align}
where $K_i$ is defined in \eqref{eq:definition_Kn}.
\end{Theorem}

\begin{Remark}[Inductive Nature]
\label{rem:inductive_proof}
It is important to observe that $\xi_i$ and therefore also $\tau_i$, only depend on $\mu_1,\dots,\mu_i$. This gives an iterative structure allowing to \enquote{add one marginal at a time} and enables us to naturally prove the Theorem by induction on $n$.
\end{Remark}

\begin{Remark}[Minimality]
Since all $\tau_i$ are such that $\left( B_{t \wedge \tau_i} \right)_{t \geq 0}$ is a uniformly integrable martingale it follows from \citet{MR0343354} that all $\tau_i$ are \textit{minimal}.
\end{Remark}

\subsection{Examples}

Examples \ref{ex:MadanYor} and \ref{ex:BHR}, respectively, show that we recover the stopping boundaries obtained by \citet{MadanYor} and \citet{Brown98themaximum}, respectively.
In particular the case $n=1$ corresponds to the solution of \citet{AzemaYor1}.

\begin{Example}[\citet{MadanYor}]
\label{ex:MadanYor}
Recall the definition of the barycentre function $b_i$ from \eqref{eq:def_barycenter}.
\citet{MadanYor} consider the \enquote{increasing mean residial value} case, i.e.
\begin{align}
b_1 \leq b_2 \leq \dots \leq b_n.
\label{eq:IMRV}
\end{align}

We will now show that our main result reproduces their result if Assumption \ref{ass:unicity_minimizers} is in place. In fact, as can be seen below, our definitions of $\xi_i$ and $K_i$, cf. \eqref{eq:definition_xi_n} and \eqref{eq:definition_Kn}, respectively, reproduce the correct stopping boundaries in the general case, showing that Assumption \ref{ass:unicity_minimizers} is not necessary, cf. also Section \ref{sec:Extensions}.
More precisely, we have
\begin{align}
\xi_i = b_i^{-1},  \quad  K_i(y) = \frac{c_i(b_i^{-1}(y))}{y-b_i^{-1}(y)} =: \mu_i^{\mathrm{HL}}([y,\infty)), \quad \quad i=1,\dots,n, 
\label{eq:def_HL_transform}
\end{align}
where $b_i^{-1}$ denotes the right-continuous inverse of $b_i$ and $\mu_i^{\mathrm{HL}}$ is the Hardy-Littlewood transform of $\mu_i$, cf. \citet{CarraroElKarouiObloj:09}.

Clearly, the claim is true for $i=1$. Let us assume that the claim holds for all $i \leq n-1$.
Now, the optimization problem for $\xi_{n}$ in \eqref{eq:definition_xi_n} becomes
\begin{align*}
\xi_{n}(y) &\in \argmin_{\zeta \leq y} \left\{  \frac{c_{n}(\zeta)}{y-\zeta} - \indicator{ \zeta>b^{-1}_{n-1}(y) } \left[ \frac{c_{n-1}(\zeta)}{y-\zeta} -  \frac{c_{n-1}(b^{-1}_{n-1}(y))}{y-b^{-1}_{n-1}(y)}  \right]   \right\} \\
						 &\in \argmin_{\zeta \leq y} \left\{ \min_{\zeta \leq b^{-1}_{n-1}(y)}   \frac{c_{n}(\zeta)}{y-\zeta}   , \min_{\zeta \geq b^{-1}_{n-1}(y)} \left( \frac{c_{n}(\zeta)}{y-\zeta} -  \left[ \frac{c_{n-1}(\zeta)}{y-\zeta} -  \frac{c_{n-1}(b^{-1}_{n-1}(y))}{y-b^{-1}_{n-1}(y)}  \right] \right)  \right\}.
\end{align*}
It is clear that the first minimum is $A_1 = \frac{c_{n}(b^{-1}_{n}(y))}{y-b^{-1}_{n}(y)}$ since $b^{-1}_{n}(y) \leq b^{-1}_{n-1}(y)$.

As for the second minimum, we set 
\begin{align*}
F(\zeta):=\frac{c_n(\zeta)}{y-\zeta}  -\left[\frac{c_{n-1}(\zeta)}{y-\zeta}                      -\frac{c_{n-1} \left(b_{n-1}^{-1}(y) \right)}{y-b_{n-1}^{-1}(y)} \right]
\end{align*} 
and we see by direct calculation that for almost all $\zeta \in \R$
\begin{alignat*}{3}
 (y-\zeta)^2 F'(\zeta)
 &=
 (b_n(\zeta)-y)\mu_n \left([\zeta,\infty)\right)
 -(b_{n-1}(\zeta)-y)\mu_{n-1}\left([\zeta,\infty)\right)
 \\
 &=
 c_n(\zeta)\frac{b_n(\zeta)-y}
                  {b_n(\zeta)-\zeta}
 -c_{n-1}(\zeta)\frac{b_{n-1}(\zeta)-y}
                       {b_{n-1}(\zeta)-\zeta}.
\end{alignat*}
By \eqref{eq:IMRV}, we conclude
therefore
\begin{align*}
 (y-\zeta)^2 F'(\zeta) \geq \left(c_n(\zeta)-c_{n-1}(\zeta)\right)
 \frac{b_{n-1}(\zeta)-y}
      {b_{n-1}(\zeta)-\zeta} \geq 0,
\end{align*}
where the last inequality follows from the non-decrease of the $\mu_i$'s in the convex order. 
Hence $F$ is non-decreasing, and it follows that it attains its minimum at the left boundary, i.e. $A_2 = \frac{c_{n}(b_{n-1}^{-1}(y))}{y-b_{n-1}^{-1}(y)} -  \left[ \frac{c_{n-1}(b_{n-1}^{-1}(y))}{y-b_{n-1}^{-1}(y)} -  \frac{c_{n-1}(b^{-1}_{n-1}(y))}{y-b^{-1}_{n-1}(y)}  \right] = \frac{c_{n}(b_{n-1}^{-1})(y)}{y-b_{n-1}^{-1}(y)}$. Consequently, by \eqref{eq:AY_minimization}, $\min \left\{ A_1, A_2  \right\}  = A_1$ and \eqref{eq:def_HL_transform} follows.
\end{Example}

\begin{Example}[\citet{Brown98themaximum}]
\label{ex:BHR}
In the case of $n=2$ our definition of $\xi_1$ and $\xi_2$ clearly recovers the stopping boundaries in the main result of \citet{Brown98themaximum}. 
However, our embedding is not as general as their embedding because we enforce Assumption \ref{ass:unicity_minimizers}, see also the discussion in Section \ref{sec:Extensions}.
\end{Example}

\begin{Example}[Locally no Constraints]
In general we have
\begin{align}
K_n(y) \leq \mu_n^{\mathrm{HL}}([y,\infty)).
\end{align}
However, if
\begin{align}
\xi_n(y) = b^{-1}_{n}(y)
\end{align}
for some $y \geq 0$ then it follows from Theorem \ref{thm:main_result} that
\begin{align}
K_n(y) = \frac{c_n(b_n^{-1}(y))}{y-b_n^{-1}(y)} = \mu_n^{\mathrm{HL}}([y,\infty)),  
\end{align}
i.e. locally at level of maximum $y$ the intermediate laws have no impact on the distribution of the terminal maximum as compared with the (one marginal) Az\'{e}ma-Yor embedding.
\end{Example}

\subsection{Properties of $\xi_n$ and $K_n$}
\label{subsec:Basic Properties of Stopping Boundaries}

Under Assumption \ref{ass:unicity_minimizers} we establish the continuity of $\xi_n$ for $n \geq 2$, cf. Lemma \ref{lem:continuity_c_power_n}, and prove monotonicity of $\xi_n$ for $n \geq 1$, cf. Lemma \ref{lem:Monotonicity_of_Stopping_Boundaries}.
In Lemma \ref{lem:ODE_K_n} we derive an ODE for $K_n$ which will be later used to identify the distribution of the maximum of the embedding from Definition \ref{def:iterated_AY}.

Let $n_1 < n_2$. Recalling Remark \ref{rem:inductive_proof} it follows that the embedding of the first $n_1$ marginals in the $n_2$-marginals embedding problem coincides with the $n_1$-marginals embedding problem. Hence it is natural to prove the Lemma by induction over the number of marginals $n$.

\begin{Lemma}[Continuity of $\xi_n$]
\label{lem:continuity_c_power_n}
Let $n \geq 2$ and let Assumption \ref{ass:unicity_minimizers} hold. 
Set 
\begin{align}
\Delta := \left\{ (x,y)\in \R \times \R_+ : x < y \right\}.
\end{align}
Then the mappings 
\begin{alignat}{5}
&c^n: \Delta \to \R, \quad &&(x,y) &&\mapsto c^n(x,y), \\
&\xi_n:\R_+ \to \R,  \quad &&\quad y     &&\mapsto \xi_n(y)  
\end{alignat}
are continuous.  
\end{Lemma}
 
\begin{proof}
We prove the claim by induction over $n$.
Let us start with the induction basis $n=1,2$. Continuity of $c^1$ is the same as continuity of $c_1$ and continuity of $c^2$ is proven by \citet{Brown98themaximum}, cf. Lemma 3.5 therein. As for continuity of $\xi_2$ we note that our Assumption \ref{ass:unicity_minimizers}(ii) precisely rules out discontinuities of $\xi_2$ as shown by \citet[Section 3.5]{Brown98themaximum}.
By induction hypothesis we assume continuity of $c^1,\dots,c^{n-1}$ and $\xi_2, \dots, \xi_{n-1}$. 

%It follows from the definition of $\imath_n$ that $\imath_n(\cdot;y)$ is left-continuous for fixed $y$.  
The only possibility that a discontinuity of $c^n$ can occur is when the index $\imath_n$ changes.
This only happens at $(x,y)=(\xi_k(y),y)$ for some $k<n$, or, in the case that $y$ is a discontinuity of $\xi_1$, at $(x,y)$ where $x \in [\xi_1(y-),\xi_1(y+) ]$. 
We prove continuity at $(x,y)$. 

Consider first the following cases:
\begin{alignat}{5}
&\text{if $x=\xi_k(y)$ }		&&\text{ then}			\quad &&x \neq \xi_j(y) \quad &&\text{for all $j \neq k$, $j < n$}, \label{eq:continuity_proof_special_case_1} \\
\text{or,} \qquad &\text{if $x \in [\xi_1(y-),\xi_1(y+) ]$} \quad &&\text{ then} \quad &&x \neq \xi_j(y) \quad &&\text{for all $j \neq 1$, $j < n$}.
\label{eq:continuity_proof_special_case_2}
\end{alignat}
Note that in case \eqref{eq:continuity_proof_special_case_2} we have from Remark \ref{rem:xi_1_discontinuous}
\begin{align}
K_1(y) = \frac{c_1(x)}{y-x} \qquad \text{for all $x \in [\xi_1(y-),\xi_1(y+)]$.}
\label{eq:non_unique_minimizer_n_1}
\end{align}

We will call a point $(x,y)$ to be \enquote{to the right of $\xi_k$} if $\xi_k(y) < x$ and \enquote{to the left of $\xi_k$} if $\xi_k(y) \geq x$.
From \eqref{eq:continuity_proof_special_case_1} and \eqref{eq:continuity_proof_special_case_2} it follows that there exists an $\epsilon > 0$ such that each point $(\tilde{x},\tilde{y})$ in the $\epsilon$-neighbourhood of $(x,y)$ is either to the left or to the right of $\xi_k$ and there are no other boundaries in this $\epsilon$-neighbourhood, in particular
\begin{align}
k = \imath_n(x_r;y_r), \qquad j = \imath_n(x_l;y_l)  = \imath_{\imath_n(x_r;y_r)}(x_r;y_r),
\label{eq:relation_indices_special_case_continuity_proof_c_n}
\end{align}
where $(x_r,y_r)$ is in the $\epsilon$-neighbourhood of $(x,y)$ and to the right of $\xi_k$ and $(x_l,y_l)$ is in the $\epsilon$-neighbourhood of $(x,y)$ and to the left of $\xi_k$.

If $x<y$, we have by induction hypothesis
\begin{alignat}{3}
c^n(x_r,y_r)  				  & =&& \hspace{5mm} c_n(x_r)-\left\{ c_k(x_r) - (y_r - x_r) K_k(y_r) \right\} \label{eq:continuity_1} \\
\xrightarrow[\text{from the right}]{(x_r,y_r)\to (x,y)} \hspace{1mm} & && \hspace{5mm} c_n(x)-\left\{ c_k(x) - (y - x) K_k(y) \right\} \nonumber  \\
							  & \stackrel[\mathclap{\eqref{eq:continuity_proof_special_case_1}-\eqref{eq:relation_indices_special_case_continuity_proof_c_n}}]{\mathclap{\eqref{eq:definition_Kn}}}{=} \hspace{2mm} && \hspace{5mm} c_n(x)-\left\{ c_k(x) - \frac{y - x}{y-x} \Big( c_k(x) - \left[ c_j(x) - (y-x)K_j(y) \right] \Big) \right\}  \nonumber  \\
							  & = \hspace{2mm} && \hspace{5mm} c_n(x)- \left[ c_j(x) - (y-x)K_j(y) \right] \nonumber  \\
							  & \stackrel{\mathclap{\eqref{eq:def_c_n}}}{=} \hspace{2mm} && \hspace{5mm} c^n(x,y) \label{eq:continuity_2}\\
							  & = \hspace{2mm} && \hspace{5mm} c_n(x)- \left[ c_j(x) - (y-x)K_j(y) \right] \nonumber \\
\xleftarrow[\text{from the left}]{(x_l,y_l)\to (x,y)}  \hspace{1mm} & && \hspace{5mm} c_n(x_l)-\left\{ c_j(x_l) - (y_l - x_l) K_j(y_l) \right\} = c^n(x_l,y_l).
\label{eq:continuity_3}
\end{alignat}
From \eqref{eq:continuity_1}, \eqref{eq:continuity_2} and \eqref{eq:continuity_3} continuity of $c^n$ follows for any sequence $(x_n,y_n) \to (x,y)$. We now extend the above argument to the situation when $x=y$ which establishes left-continuity of $c^n$ at $(y,y)$. In this case we have $x=\xi_k(y)=y$. For this to hold we must have $c_k(y) = c_j(y)$. Using boundedness of $K_i$ for $i<n$ shows that \eqref{eq:continuity_2} and \eqref{eq:continuity_3} converge to each other. 

To relax \eqref{eq:continuity_proof_special_case_1} and \eqref{eq:continuity_proof_special_case_2} we successively write out $K_k, K_j, \dots,$ until the assumption of the first case holds true and then, successively, apply the special case.

It remains to prove continuity of $\xi_n$ which we prove by contradiction. Assume there exist $\epsilon>0$ and $y \geq 0$ such that for all $\delta>0$ there exists a $y' \in (y,y+\delta)$ such that $| \xi_n(y) - \xi_n(y') |>\epsilon $. 
By \eqref{eq:lower_bound_for_xi_n} the limit of $\xi_n(y')$ as $y' \downarrow y$ exists at least along some subsequence and we denote it by $\tilde{\xi}_n$. By assumption $\tilde{\xi}_n \neq \xi_n(y)$.

Consider first the case that $\xi_n(y)<y$ and $\tilde{\xi}_n < y$.
Using continuity of $c^n$ we deduce
$\frac{c^n(\xi_n(y'),y')}{y'-\xi_n(y')} \to \frac{c^n(\tilde{\xi}_n,y)}{y-\tilde{\xi}_n}$ as $y' \to y$.

Now, if 
\begin{align}
 \frac{c^n(\tilde{\xi}_n ,y)}{y -\tilde{\xi}_n } \neq \frac{c^n(\xi_n(y),y)}{y-\xi_n(y)}  
 \label{eq:c_n_different}
\end{align}  
then we obtain a contradiction to the optimality of either $\xi_n(y)$ or some $\xi_n(y')$ for $y'$ close enough to $y$ by continuity of $c^n$. 
If
\begin{align}
 \frac{c^n(\tilde{\xi}_n,y)}{y-\tilde{\xi}_n} = \frac{c^n(\xi_n(y),y)}{y-\xi_n(y)} 
  \label{eq:c_n_same}
\end{align}
we obtain a contradiction to Assumption \ref{ass:unicity_minimizers}(ii).

We now consider the case that either $\xi_n(y) = y$ or $\tilde{\xi}_n = y$.
The case $\xi_n(y)<y$ and $\tilde{\xi}_n = y$ is ruled out by condition \eqref{eq:assumption_strictly_ordered_calls} from Assumption \ref{ass:unicity_minimizers}(ii): Indeed, for the sequence $\left( K_n(y') = \frac{c^n(\xi_n(y'),y')}{y'-\xi_n(y')} \right)$ to be bounded we must have $c^n(\xi_n(y'),y') \to 0$. Recalling the left-continuity of $c^n$ at $(y,y)$ implies $c_n(y) = c_{\imath_n(y;y)}(y)$.

The case $\xi_n(y)=y$ and $\tilde{\xi}_n < y$ follows as above by distinguishing the cases \eqref{eq:c_n_different} and \eqref{eq:c_n_same} and by recalling \eqref{eq:extension_c_n} and the left-continuity of $c^n$ at $(y,y)$.
\end{proof}

\begin{Lemma}[Monotonicity of $\xi_n$]
\label{lem:Monotonicity_of_Stopping_Boundaries}

Let $n\in \N$ and let Assumption \ref{ass:unicity_minimizers} hold. 
Then 
\begin{align}
\xi_n:\R_+ \to \R, \quad y \mapsto \xi_n(y) \qquad \text{is non-decreasing.}
\label{eq:monotonicity_xi}
\end{align} 
\end{Lemma}

\begin{proof}
The claim for $n=1,2$ follows from \citet{Brown98themaximum}.
Assume by induction hypothesis that we have proven monotonicity of $\xi_1,\dots,\xi_{n-1}$.

We follow closely the arguments of \citet[Lemma 3.2]{Brown98themaximum}.
Since $\xi_n$ is continuous it is enough to prove monotonicity at almost every $y \geq 0$.
The set of $y$'s which are a discontinuity of $\xi_1$ is a null-set, and hence we can exclude all such $y$'s. 
In the following we fix a $y$ where $\xi_1,\dots,\xi_n$ are continuous. 

We will first consider the case when $\xi_n(y) \neq \xi_j(y)$ for all $j<n$. By continuity of $\xi_n$ it follows that there is an $\epsilon>0$ such that 
\begin{align}
\text{$\xi_n(\tilde{y}) \neq \xi_j(\tilde{y})$ and $\ell := \jmath_n(y) = \jmath_n(\tilde{y})$ for all $\tilde{y} \in (y-\epsilon, y+\epsilon)$ and $j<n$,}
\label{eq:no_boundary_crossing_montonicity_proof}
\end{align}
and furthermore
\begin{align}
(\xi_n(\tilde{y}),\tilde{y}) &\in (\xi_n(y)-\epsilon, \xi_n(y)+\epsilon) \times (y-\epsilon, y+\epsilon).
\label{eq:epsilon_ball_montonicity_proof}
\end{align}

Let $l_1$ denote a supporting tangent to $c^n(\cdot,y)$ at $\xi_n(y)$ which goes through the $x$-axis at $y$, i.e.  
\begin{align*}
l_1(x) = c^n(\xi_n(y),y) + (x-\xi_n(y))(D - K_{\ell}(y)),
%\label{eq:expression_1_for_l_1}
\end{align*}
where $D$ lies between the left- and right-derivatives of $c_n - c_{\ell}$ at $\xi_n(y)$.
%Alternatively, we can express $l_1$ as 
%\begin{align}
%l_1(x) = \frac{y-x}{y-\xi_n(y)} c^n(\xi_n(y),y)  
%\end{align}
%which implies that 
Using that $l_1(y) = 0$ we can write
\begin{align*}
D - K_{\ell}(y) =  -\frac{c^n(\xi_n(y),y)}{y-\xi_n(y)} \stackrel{\eqref{eq:def_c_n}}{=} - \frac{c_n(\xi_n(y)) -c_{\ell}(\xi_n(y))}{y-\xi_n(y)} - K_{\ell}(y)
\end{align*}
and thus by \eqref{eq:ordered_calls}
\begin{align}
D \leq 0.% \qquad \text{i.e.} \quad -c'_{\jmath_n(y)}(\xi_n(y)\pm) \leq -c'_n(\xi_n(y)\pm).
\label{eq:stochastic_ordering_at_xi}
\end{align}
We also have
\begin{align}
l_1(y+\delta) = \delta(D-K_{\ell}(y)).
\label{eq:l_1(y+delta)}
\end{align}

Choose $\delta \in (0, \epsilon)$ sufficiently small. Our goal is to prove $\xi_n(y+\delta) \geq \xi_n(y)$. 
Recall that $\xi_n(y+\delta)$ is determined from $y+\delta$ and $c^n(\cdot,y+\delta)$ only. Since we know that $\xi_n(y+\delta) \in (\xi_n(y)-\epsilon, \xi_n(y)+\epsilon) := I$ it will turn out to be enough to look at $c^n(x,y+\delta)$ only for $x \in (\xi_n(y)-\epsilon, \xi_n(y)+\epsilon)$. 
For such an $x$ we have
\begin{align}
c^n(x,y+\delta) - c^n(x,y) \stackrel{\eqref{eq:def_c_n}}{=}  
\left( y+\delta-x \right) K_{\ell}(y+\delta)  - \left( y-x \right) K_{\ell}(y).
\label{eq:differnece_c_n}
\end{align}

Let $l_2$ be the supporting tangent to $c^n(\cdot,y+\delta) - c^n(\cdot,y)$ at $\xi_n(y)$, i.e.
\begin{align*}
l_2(x) = c^n(\xi_n(y),y+\delta) - c^n(\xi_n(y),y) + (x-\xi_n(y))(K_{\ell}(y) - K_{\ell}(y+\delta)).
\end{align*}
Hence,
\begin{align}
l_1(y+\delta) + l_2(y+\delta) \hspace{2mm} \stackrel{\mathclap{\eqref{eq:l_1(y+delta)}}}{=}& \hspace{5mm}  \delta(D-K_{\ell}(y)) \nonumber \\ 
								& + c^n(\xi_n(y),y+\delta) - c^n(\xi_n(y),y) \nonumber \\
								& + (y + \delta - \xi_n(y))(K_{\ell}(y) - K_{\ell}(y+ \delta)) \nonumber \\
							   \stackrel{\mathclap{\eqref{eq:differnece_c_n}}}{=}& \hspace{5mm} \delta D \leq 0.
\label{eq:tangent_negative}	
\end{align}

Now, since $c^n(\cdot,y+\delta) - c^n(\cdot,y)$ is linear (and therefore convex) in the domain $I$, $l_1 + l_2$ is a supporting tangent to $c^n(\cdot,y+\delta)$ at $\xi_n(y)$, i.e.
\begin{alignat}{3}
(l_1 + l_2)(x)        &\leq c^n(x,y+\delta) \quad &&\text{for $x \in I$}, \\
(l_1 + l_2)(\xi_n(y)) &= c^n(\xi_n(y),y+\delta).\quad &&
\end{alignat}
Recall that $\xi_n(y+\delta)$ is determined as the $x$-value where the supporting tangent to $c^n(\cdot,y+\delta)$ which passes the $x$-axis at $y+\delta$ touches $c^n(\cdot,y+\delta).$
Next we exploit the fact that $\xi_n(y+\delta) \in I$ which implies that we only need to show that $\xi_n(y+\delta) \not\in (\xi_n(y)-\epsilon,\xi_n(y))$.
Indeed, this follows from \eqref{eq:tangent_negative} which yields that any supporting tangent to $c^n(\cdot,y+\delta)$ at some $\zeta \in (\xi_n(y)-\epsilon, \xi_n(y))$ must be below the $x$-axis when evaluated at $y+\delta$. We refer to \citet[Fig.7]{Brown98themaximum} for a graphical illustration of this fact.

Now we relax the assumption \eqref{eq:no_boundary_crossing_montonicity_proof}. Assume that there exists a $\delta>0$ such that $\xi_n(y)>\xi_n(y+\delta)$. We derive a contradiction to the special case as follows. Set $y_0:=y$ and $y_n:=y+\delta$. Recall that $\xi_n$ is continuous. Now we can choose $y_0<y_1<\dots<y_{n-1}<y_n$ such that $\xi_n(y_0)>\xi_n(y_1)>\dots>\xi_n(y_{n-1})>\xi_n(y_n)$. Set $x_i:=\xi_n(y_i)$, $i=0,\dots,n$. 
Observe that by monotonicity of $\xi_k$, $k<n$ the graph of $\xi_k$ intersects with at most one rectangle $(x_i,x_{i-1}) \times (y_{i-1},y_i)$, $i=1,\dots,n$. Consequently, there must exist at least one integer $j$ such that the rectangle $R:=(x_j,x_{j-1}) \times (y_{j-1},y_j)$ is disjoint with the graph of every $\xi_k$, $k<n$. By construction and continuity of $y \mapsto \xi_n(y)$ $R$ is not disjoint with the graph of $\xi_n$.
Inside this rectangle $R$ the conditions of the special case \eqref{eq:no_boundary_crossing_montonicity_proof} are satisfied. 
Recalling that $\xi_n(y_j)=x_j<x_{j-1}=\xi_n(y_{j-1})$ and by continuity of $y \mapsto \xi_n(y)$, we can find two points $s_1<s_2$ such that $z_1=\xi_n(s_1)>\xi_n(s_2)=z_2$ and $(z_1,s_1) \in R, (z_2,s_2) \in R$. This is a contradiction.
\end{proof}

\begin{Lemma}[ODE for $K_n$]
\label{lem:ODE_K_n}
Let $n\in \N$ and let Assumption \ref{ass:unicity_minimizers} hold. 
Then 
\begin{align}
y \mapsto K_n(y) \qquad \text{is absolutely continuous and non-increasing.}
\end{align}

If we assume in addition that the embedding property of Theorem \ref{thm:main_result} is valid for the first $n-1$ marginals then for almost all $y \geq 0$ we have: 

If $\xi_n(y)<y$ then
\begin{align}
K'_n(y)  + \frac{K_n(y)}{y-\xi_n(y)} = K'_{\jmath_n(y)}(y) + \frac{K_{\jmath_n(y)}(y)}{y-\xi_n(y)}
\label{eq:ODE_K_n}
\end{align}
where $K'_j$ denotes the derivative of $K_j$ which exists for almost all $y \geq 0$ and $j=1,\dots,n$.

If $\xi_n(y)=y$ then
\begin{align}
K_n(y+) = K_{\jmath_n(y)}(y+).
\label{eq:ODE_K_n_2}
\end{align}
\end{Lemma}
\begin{proof}
The proof is reported in the Appendix \ref{sec:appendix_1}.
\end{proof}

\section{Proof of the Main Result}
\label{sec:iAY}

In this Section we prove the main result, Theorem \ref{thm:main_result}.
The key step is the identification of the distribution of the maximum, cf. Proposition \ref{prop:Maximizing the Maximum}.
%, for which there is equality in \eqref{eq:lem:Trajectorial_Inequality_Unordered_Case}. Thus, we want to prove that there is \enquote{no duality gap}.

Let $n \in \N$. For convenience we set 
\begin{align}
M_0 := 0, \qquad M_i := B_{\tau_i}, \qquad i=1,\dots,n,
\end{align}
where $\tau_i$ is defined in Definition \ref{def:iterated_AY}.

\subsection{Basic Properties of the Embedding}
 
Our first result shows that there is a \enquote{strong relation} between $M$ and $\bar{M}$. 
 
\begin{Lemma}[Relations Between $M$ and $\bar{M}$]
\label{lem:Key_Properties_of_the_Stopping_Rule}
Let $n\in \N$ and let Assumption \ref{ass:unicity_minimizers} hold. 
%It follows directly from Definition \ref{def:iterated_AY} that
%\begin{align}
%M_i \leq \xi_i(\bar{M}_i), \quad i=1,\dots,n.
%\end{align}
Then the following implications hold. 
\begin{align}
&M_n   >  \xi_n(y)  \quad \Longrightarrow \quad \bar{M}_{n}  \geq y,   \label{eq:impl5} \\
&M_n   \geq  \xi_n(y)  \quad \Longrightarrow \quad \bar{M}_{n}  \geq y  \quad \text{if $\xi_n$ is strictly increasing at $y$.} \label{eq:impl3} 
\end{align}
For $y \geq 0$ such that $\jmath_n(y) \neq 0$ we have
\begin{alignat}{7}
&M_{\jmath_n(y)} \geq \xi_{n}(y) > \xi_{\jmath_n(y)}(y) &&  \quad && \Longrightarrow && \qquad && M_n   \geq  \xi_n(y),  \label{eq:impl1} \\
&\bar{M}_{\jmath_n(y)} < y,\hspace{2mm} \bar{M}_{n} \geq y && \quad && \Longrightarrow && \qquad && M_n   \geq  \xi_n(y),  \label{eq:impl2} \\
&\bar{M}_{\jmath_n(y)} \geq y,\hspace{2mm} M_{\jmath_n(y)} < \xi_{n}(y) &&  \quad && \Longrightarrow && \qquad && M_{n}  <  \xi_{n}(y). \label{eq:impl4} 
\end{alignat}
If $\xi_n$ is strictly increasing at $y \geq 0$ and $\jmath_n(y) = 0$ then the following holds. 
\begin{align}
M_n \geq \xi_n(y) \quad \Longleftrightarrow \quad \bar{M}_n \geq y .
\label{eq:equivalence_M_M_bar}
\end{align}
\end{Lemma}

\begin{proof}
Write $\jmath = \jmath_n$. 
We have  
\begin{align*}
\xi_{\jmath(y)}(y) < \xi_n(y) \leq \xi_i(y), \quad i=\jmath(y)+1,\dots,n .
\label{eq:locally_ordered_boundaries}
\end{align*}

In the following we are using continuity and monotonicity of $\xi_1,\dots,\xi_n$, cf. Lemma \ref{lem:continuity_c_power_n} and \ref{lem:Monotonicity_of_Stopping_Boundaries}.

\textit{Case $\jmath(y) \neq 0$.}
As for implication \eqref{eq:impl5} assume that $M_n > \xi_n(y)$ and $\bar{M}_{n} < y$ holds. In this case $M_n$ cannot be at the boundary $\xi_n$. It has to be at a boundary point $\xi_j(y')$ for some $j<n$ and some $y'<y$. However, this cannot be true because $\xi_n(y') \leq \xi_n(y) <  \xi_j(y')$ and hence case \eqref{eq:embedding1} of the definition of $\tau_n$ would have been triggered.  

Implication \eqref{eq:impl3} follows by the same arguments as for implication \eqref{eq:impl5}.

Implication \eqref{eq:impl1} now follows from implication \eqref{eq:impl5} applied for $\jmath(y)$ and the fact that either $M_n = M_{\jmath(y)}$ (case \eqref{eq:embedding3}) or $M$ moves to a point at the boundary $\xi_i(y') \geq \xi_n(y)$ for some $i=\jmath(y)+1,\dots,n$, $y' \geq y$ (case \eqref{eq:embedding1}).

Implication \eqref{eq:impl2} holds because if $M$ increases its maximum at time $\jmath(y)$, which is $< y$, to some $y' \geq y$ at time $n$, it will hit a boundary point $\xi_i(y') \geq \xi_n(y)$ for some $i=\jmath(y)+1,\dots,n$.

Implication \eqref{eq:impl4} holds because from $\bar{M}_{\jmath(y)} \geq y$ and $M_{\jmath(y)} < \xi_{n}(y)$ it follows that $M_{\jmath(y)}=\xi_i(y')<\xi_n(y)\leq \xi_j(y') $ for some $i \leq \jmath(y)$, $y' \geq y$, $j>\jmath(y)$. From this it follows that $M$ will stay where it is until time $n$, cf. case \eqref{eq:embedding3}.

\textit{Case $\jmath(y) = 0$.}
The condition $M_n \geq \xi_n(y)$ implies in a similar fashion as in  \eqref{eq:impl3} that $\bar{M}_n \geq y$ holds.
Conversely, assume that $\bar{M}_n \geq y$ holds. In this case $M_n$ must be at a boundary point $\xi_i(y') \geq \xi_n(y)$ for some $i=1,\dots,n$, $y' \geq y$.
\end{proof}

As an application of Lemma \ref{lem:Key_Properties_of_the_Stopping_Rule} we obtain the following result.

\begin{Lemma}[Contributions to the Maximum]
\label{lem:Contributions to the Maximum}
Let $n\in \N$ and let Assumption \ref{ass:unicity_minimizers} hold. 
Assume $\xi_n$ is strictly increasing at $y \geq 0$. 

Then, if $\jmath_n(y) \neq 0$ 
\begin{align}
\PP \left[ \bar{M}_n \geq y \right]  = \PP \left[ M_n \geq \xi_n(y) \right] - \PP\left[ M_{\jmath_n(y)} \geq \xi_n(y) \right] + \PP\left[ \bar{M}_{\jmath_n(y)} \geq y \right] %K_{\jmath(y)}(y).
%\label{eq:decomposing_the _maximum}
\end{align}
and if $\jmath_n(y) = 0$ 
\begin{align}
\PP \left[ \bar{M}_n \geq y \right]  = \PP \left[ M_n \geq \xi_n(y) \right].
\end{align}
\end{Lemma}
\begin{proof}
Write $\jmath = \jmath_n$. 

\textit{Case $\jmath(y) \neq 0$.}
Firstly, let us compute
\begin{align*}
 & \quad \PP \left[ \bar{M}_n \geq y \right] - \PP \left[ M_n \geq \xi_n(y) \right] \\ \stackrel{\mathclap{\eqref{eq:impl3}}}{=} & \quad \PP \left[ \bar{M}_n \geq y \right] - \PP \left[ M_n \geq \xi_n(y),\bar{M}_n \geq y \right] = \PP \left[ \bar{M}_n \geq y , M_n < \xi_n(y) \right] \\
=& \quad \PP \left[ \bar{M}_n \geq y , M_n < \xi_n(y), \bar{M}_{\jmath(y)} \geq y \right] + \PP \left[ \bar{M}_n \geq y , M_n < \xi_n(y), \bar{M}_{\jmath(y)} < y \right]  \nonumber	\\
\stackrel[ \mathclap{\eqref{eq:impl2}} ]{\mathclap{\eqref{eq:impl1}}}{=} & \quad	\PP \left[  M_n < \xi_n(y), \bar{M}_{\jmath(y)} \geq y, M_{\jmath(y)} < \xi_n(y) \right].
\end{align*}

Secondly, let us compute
\begin{align*}
 & \quad \PP\left[ \bar{M}_{\jmath(y)} \geq y \right] - \PP\left[ M_{\jmath(y)} \geq \xi_n(y) \right] \\ 
=& \quad \PP \left[  \bar{M}_{\jmath(y)} \geq y, M_{\jmath(y)} \geq \xi_n(y) \right]  +  \PP \left[  \bar{M}_{\jmath(y)} \geq y, M_{\jmath(y)} <    \xi_n(y) \right]  - \PP \left[   M_{\jmath(y)} \geq  \xi_n(y) \right] \\
 \stackrel{\mathclap{\eqref{eq:impl5}}}{=} &  \quad \PP \left[  \bar{M}_{\jmath(y)} \geq y, M_{\jmath(y)} <    \xi_n(y) \right]    \stackrel{\eqref{eq:impl4}}{=} \PP \left[  M_n < \xi_n(y), \bar{M}_{\jmath(y)} \geq y, M_{\jmath(y)} < \xi_n(y) \right].
				%\label{eq:rhs}			
\end{align*}
Comparing these two equations yields the claim.

\textit{Case $\jmath(y) = 0$.} 
The claim follows from \eqref{eq:equivalence_M_M_bar}.
\end{proof}

\subsection{Law of the Maximum}

Our next goal is to identify the distribution of $M_n$. We will achieve this by deriving an ODE for  $\PP\left[ \bar{M}_n \geq \cdot \right]$ using excursion theoretical results, cf. Lemma \ref{lem:ODE_for_the_Maximum}, and link it to the ODE satisfied by $K_n$, cf. Lemma \ref{lem:ODE_K_n}.

\begin{Lemma}[ODE for the Maximum] %$\PP\left[ \bar{M}_n \geq \cdot \right]$
\label{lem:ODE_for_the_Maximum}
Let $n\in \N$ and let Assumption \ref{ass:unicity_minimizers} hold. 
Then the mapping
\begin{align*}
y \mapsto \PP \left[ \bar{M}_n \geq y \right]
\end{align*}
is absolutely continuous and for almost all $y \geq 0$ we have:

If $\xi_n(y)<y$ then
\begin{align}
\frac{\partial \PP \left[ \bar{M}_{n} \geq y \right]}{\partial y}  +  \frac{\PP \left[ \bar{M}_n \geq y \right]}{y-\xi_n(y)}  
= \frac{\PP \left[ \bar{M}_{\jmath_n(y)} \geq y \right]}{y-\xi_n(y)} +  \restr{\frac{\partial \PP \left[ \bar{M}_{j} \geq y \right]}{\partial y}}{j=\jmath_n(y)}.
\label{eq:ODE_law_maximum}
\end{align}

If $\xi_n(y)=y$ then
\begin{align}
\PP \left[ \bar{M}_{n} > y \right] = \PP \left[ \bar{M}_{\jmath_n(y)} > y \right].
\label{eq:ODE_law_maximum_2}
\end{align}
\end{Lemma}

\begin{proof}
Write $\jmath = \jmath_n$.  
We exclude all $y>0$ which are a discontinuity of $\xi_1$. This is clearly a null-set.

The cases $n=1,2$ are true by \citet{Brown98themaximum}. Assume by induction hypothesis that we have proven the claim for $i=1,\dots,n-1$.

If $\xi_n(y)=y$ then it is clear from the definition of the embedding, cf. Definition \ref{def:iterated_AY}, that  
\begin{align}
\bar{M}_n >y \qquad \Longleftrightarrow \qquad \bar{M}_{\jmath} >y.
\end{align}

\textit{Case $\jmath(y) \neq 0$.}
For $\delta>0$ we have
\begin{alignat}{3}
  &  && \PP\left[ \bar{M}_n \geq y+\delta, \bar{M}_{\jmath(y)} < y+\delta \right] - \PP\left[ \bar{M}_n \geq y , \bar{M}_{\jmath(y)} < y  \right] \nonumber \\
= &  && \PP\left[ \bar{M}_n \geq y+\delta, \bar{M}_{\jmath(y)} < y  \right] - \PP\left[ \bar{M}_n \geq y , \bar{M}_{\jmath(y)} < y  \right] \label{eq:right_derivative_aux} \\
 & + && \underbrace{\PP\left[ \bar{M}_n \geq y+\delta, y < \bar{M}_{\jmath(y)} < y+\delta \right].}_{=0 \text{ for $\delta>0$ small enough by definition of $\jmath(y)$ and continuity of $\xi_i$}} \nonumber
\end{alignat}

%Similarly, we have
%\begin{alignat}{3}
%  &  && \PP\left[ \bar{M}_n > y-\delta, \bar{M}_{\jmath(y)} < y-\delta \right] - \PP\left[ \bar{M}_n > y , \bar{M}_{\jmath(y)} < y  \right] \nonumber \\
%= &  && \PP\left[ \bar{M}_n > y-\delta, \bar{M}_{\jmath(y)} < y-\delta  \right] - \PP\left[ \bar{M}_n > y , \bar{M}_{\jmath(y)} < y-\delta  \right] \label{eq:left_derivative_aux} \\
% & - && \underbrace{\PP\left[ \bar{M}_n > y, y-\delta < \bar{M}_{\jmath(y)} < y \right]}_{=0 \text{ for $\delta>0$ small enough by \eqref{}}}. \nonumber
%\end{alignat}

For $r>0$ we define
\begin{align*}
&\bar{\xi}_{j}(r) := \max_{k:j \leq k \leq n}\left\{\xi_k(r):\xi_k(y)=\xi_n(y) \right\}, \\
&\underline{\xi}_{j}(r) := \min_{k:j \leq k \leq n}\left\{\xi_k(r):\xi_k(y)=\xi_n(y) \right\}
\end{align*}
and note that 
\begin{align}
\bar{\xi}_{\jmath(y)}(r) \to \xi_n(y), \qquad \underline{\xi}_{\jmath(y)}(r) \to  \xi_n(y) \qquad \text{as $r \to y$}
\label{eq:convergence_effective_stopping_boundaries}
\end{align}
by continuity of $\xi_i$ at $y$ for $i=1,\dots,n$.

Let $\delta>0$.
We have by excursion theoretical results, cf. e.g. \citet{rogers89g},
\begin{equation}
\begin{split}
     &\hspace{1mm}\PP\left[ \bar{M}_n \geq y , \bar{M}_{\jmath(y)} < y  \right] \exp \left(- \int_{y}^{y+\delta} \frac{\dd r}{r- \bar{\xi}_{\jmath(y)}(r)} \right) \\
\leq &\hspace{1mm}\PP\left[ \bar{M}_n \geq y+\delta, \bar{M}_{\jmath(y)} < y  \right] \\
\leq &\hspace{1mm}\PP\left[ \bar{M}_n \geq y , \bar{M}_{\jmath(y)} < y  \right] \exp \left(- \int_{y}^{y+\delta} \frac{\dd r}{r- \underline{\xi}_{\jmath(y)}(r)} \right).
\end{split}
\label{eq:right_derivative_lower_upper_bound}
\end{equation}
%and similarly
%\begin{equation}
%\begin{split}
%     &\hspace{1mm}\PP\left[ \bar{M}_n > y-\delta , \bar{M}_{\jmath(y)} < y-\delta  \right] \exp \left(- \int_{y-\delta}^{y} \frac{\dd r}{r- \bar{\xi}_{\jmath(y)}(r)} \right) \\
%\leq &\hspace{1mm}\PP\left[ \bar{M}_n > y, \bar{M}_{\jmath(y)} < y-\delta  \right] \\
%\leq &\hspace{1mm}\PP\left[ \bar{M}_n > y-\delta, \bar{M}_{\jmath(y)} < y-\delta \right] \exp \left(- \int_{y-\delta}^{y} \frac{\dd r}{r- \underline{\xi}_{\jmath(y)}(r)} \right)
%\end{split}
%\label{eq:left_derivative_lower_upper_bound}
%\end{equation}
 
Now we compute for $y$ such that $\xi_n(y)<y$
\begin{alignat}{5}
& && && \frac{\PP\left[ \bar{M}_n \geq y+\delta, \bar{M}_{\jmath(y)} < y+\delta \right] - \PP\left[ \bar{M}_n \geq y , \bar{M}_{\jmath(y)} < y  \right]}{\delta} \nonumber \\
& \quad \stackrel{\mathclap{\eqref{eq:right_derivative_aux},\eqref{eq:right_derivative_lower_upper_bound}}}{\leq } \hspace{10mm}&& && \PP\left[ \bar{M}_n \geq y , \bar{M}_{\jmath(y)} < y  \right] \frac{ \exp \left(- \int_{y}^{y+\delta} \frac{\dd r}{r- \underline{\xi}_{\jmath(y)}(r)} \right)-1 }{\delta} \nonumber \\
&\xrightarrow[\text{as $\delta \downarrow 0$}]{\text{by \eqref{eq:convergence_effective_stopping_boundaries}}}   && - && \frac{ \PP\left[ \bar{M}_n \geq y , \bar{M}_{\jmath(y)} < y  \right] }{y-\xi_n(y)} 
\label{eq:upper_bound_right_derivative}
\end{alignat}
and
\begin{alignat}{5}
& && && \frac{\PP\left[ \bar{M}_n \geq y+\delta, \bar{M}_{\jmath(y)} < y+\delta \right] - \PP\left[ \bar{M}_n \geq y , \bar{M}_{\jmath(y)} < y  \right]}{\delta} \nonumber \\
& \quad \stackrel{\mathclap{\eqref{eq:right_derivative_aux},\eqref{eq:right_derivative_lower_upper_bound}}}{\geq } \hspace{10mm}&& && \PP\left[ \bar{M}_n \geq y , \bar{M}_{\jmath(y)} < y  \right] \frac{ \exp \left(- \int_{y}^{y+\delta} \frac{\dd r}{r- \bar{\xi}_{\jmath(y)}(r)} \right)-1 }{\delta} \nonumber \\
&\xrightarrow[\text{as $\delta \downarrow 0$}]{\text{by \eqref{eq:convergence_effective_stopping_boundaries}}}   && - &&  \frac{\PP\left[ \bar{M}_n \geq y , \bar{M}_{\jmath(y)} < y  \right]}{y-\xi_n(y)} .
\label{eq:lower_bound_right_derivative}
\end{alignat}

Hence, from \eqref{eq:upper_bound_right_derivative} and \eqref{eq:lower_bound_right_derivative} it follows that the right-derivative of 
\begin{align}
y \mapsto \restr{ \PP\left[ \bar{M}_n \geq y , \bar{M}_{j} < y  \right] }{j=\jmath(y)}
\label{eq:mixed_term_1_excursion_proof}
\end{align}
exists. 
Similar arguments for $\delta<0$ show that the left-derivative exists and is the same as the right-derivative.
Local Lipschitz continuity of \eqref{eq:mixed_term_1_excursion_proof} then follows from \eqref{eq:upper_bound_right_derivative} and \eqref{eq:lower_bound_right_derivative}.

Observe the obvious equality
\begin{align}
\PP \left[ \bar{M}_n \geq y \right] = \PP \left[ \bar{M}_{j} \geq y \right] + \PP \left[ \bar{M}_n \geq y, \bar{M}_{j} < y \right]
\label{eq:trivial_decomposition_max_n}
\end{align}
Taking $j=\jmath(y)$ in \eqref{eq:trivial_decomposition_max_n} and fixing it, we conclude by induction hypothesis that $y \mapsto  \PP\left[ \bar{M}_n > y \right]$ is locally Lipschitz continuous and hence absolutely continuous and its derivative reads
\begin{align*}
\frac{\partial \PP \left[ \bar{M}_{n} \geq y \right]}{\partial y} = 
 \restr{\frac{\partial \PP \left[ \bar{M}_{j} \geq y \right]}{\partial y}}{j=\jmath_n(y)}
+  \frac{\PP \left[ \bar{M}_{\jmath_n(y)} \geq y \right] - \PP \left[ \bar{M}_n \geq y \right]}{y-\xi_n(y)}.
\end{align*}

%Similarly, 
%\begin{align*}
%& \frac{\PP\left[ \bar{M}_n > y-\delta, \bar{M}_{\jmath(y)} < y-\delta \right] - \PP\left[ %\bar{M}_n > y , \bar{M}_{\jmath(y)} < y  \right]}{-\delta} \\
%\leq & \PP\left[ \bar{M}_n > y-\delta , \bar{M}_{\jmath(y)} < y-\delta  \right] \frac{ \exp \left(- \int_{y-\delta}^{y} \frac{\dd r}{r- \underline{\xi}_{\jmath(y)}(r)} \right)-1 }{\delta} \\
%\overset{\text{as $\delta \to 0$}}{\longrightarrow} & - \frac{ \PP\left[ \bar{M}_n > y , \bar{M}_{\jmath(y)} < y  \right] }{y-\xi_n(y)} 
%\end{align*}
%and
%\begin{align*}
%& \frac{\PP\left[ \bar{M}_n > y-\delta, \bar{M}_{\jmath(y)} < y-\delta \right] - \PP\left[ \bar{M}_n > y , \bar{M}_{\jmath(y)} < y  \right]}{-\delta} \\
%\geq & \PP\left[ \bar{M}_n > y-\delta , \bar{M}_{\jmath(y)} < y-\delta  \right] \frac{ \exp \left(- \int_{y-\delta}^{y} \frac{\dd r}{r- \bar{\xi}_{\jmath(y)}(r)} \right)-1 }{\delta} \\
%\overset{\text{as $\delta \to 0$}}{\longrightarrow} &  - \frac{\PP\left[ \bar{M}_n > y , \bar{M}_{\jmath(y)} < y  \right]}{y-\xi_n(y)} .
%\end{align*}

\textit{Case $\jmath(y) = 0$.}
For $\delta>0$ we have by excursion theoretical results
\begin{alignat}{3}
  &  && \PP\left[ \bar{M}_n \geq y+\delta, \bar{M}_{1} < y+\delta \right] - \PP\left[ \bar{M}_n \geq y , \bar{M}_{1} < y  \right]  \nonumber \\
= &  && \PP\left[ \bar{M}_n \geq y+\delta, \bar{M}_{1} < y+\delta  \right] - \PP\left[ \bar{M}_n \geq y+\delta, \bar{M}_{1} < y  \right] \nonumber \\
 & + &&  \PP\left[ \bar{M}_n \geq y+\delta, \bar{M}_{1} < y   \right] - \PP\left[ \bar{M}_n \geq y, \bar{M}_{1} < y  \right] \nonumber \\
 \leq &  && \int_{y}^{y+\delta} \PP\left[ \bar{M}_1 \in \dd s \right] \frac{ \left( \xi_1(s) - \underline{\xi}_1(s) \right)^+}{s-\underline{\xi}_1(s)} \exp \left( -\int_{s}^{y+\delta} \frac{\dd r}{r-\underline{\xi}_1(r)}\right) \nonumber \\
 & +  &&   \PP\left[ \bar{M}_n \geq y, \bar{M}_{1} < y \right] \left[ \exp \left( -\int_{y}^{y+\delta} \frac{\dd r}{r - \underline{\xi}_1(r)} \right) - 1 \right].
 \label{eq:j_0_aux_derivative_right}
\end{alignat}

Similarly, we have
\begin{alignat}{3}
  &  && \PP\left[ \bar{M}_n \geq y+\delta, \bar{M}_{1} < y+\delta \right] - \PP\left[ \bar{M}_n \geq y , \bar{M}_{1} < y  \right]  \nonumber \\
\geq &  && \int_{y}^{y+\delta} \PP\left[ \bar{M}_1 \in \dd s \right] \frac{ \left( \xi_1(s) - \bar{\xi}_1(s) \right)^+}{s-\bar{\xi}_1(s)} \exp \left( -\int_{s}^{y+\delta} \frac{\dd r}{r-\bar{\xi}_1(r)}\right) \nonumber \\
 & +  &&   \PP\left[ \bar{M}_n \geq y, \bar{M}_{1} < y \right] \left[ \exp \left( -\int_{y}^{y+\delta} \frac{\dd r}{r - \bar{\xi}_1(r)} \right) - 1 \right].
\label{eq:j_0_aux_derivative_left}
\end{alignat}

From \eqref{eq:j_0_aux_derivative_right} and \eqref{eq:j_0_aux_derivative_left} it follows that the right-derivative of
\begin{align}
y \mapsto \PP \left[ \bar{M}_n \geq y, \bar{M}_{1} < y \right]
\label{eq:mixed_term_2_excursion_proof}
\end{align}
exists. 
Similar arguments for $\delta<0$ show that the left-derivative exists and is the same as the right-derivative.
Local Lipschitz continuity of \eqref{eq:mixed_term_2_excursion_proof} then follows from \eqref{eq:j_0_aux_derivative_right} and \eqref{eq:j_0_aux_derivative_left}.
Now we can conclude from \eqref{eq:trivial_decomposition_max_n}--\eqref{eq:j_0_aux_derivative_left} applied with $j=1$ that
$y \mapsto \PP \left[ \bar{M}_n \geq y \right]$ is locally Lipschitz continuous and hence absolutely continuous and its derivative reads
\begin{align*}
&\frac{\partial \PP \left[ \bar{M}_{n} \geq y \right]}{\partial y} \\
\stackrel{\eqref{eq:convergence_effective_stopping_boundaries}}{=}& \frac{\partial \PP \left[ \bar{M}_{1} \geq y \right]}{\partial y} - \frac{\partial \PP \left[ \bar{M}_{1} \geq y \right]}{\partial y} \frac{\left(  \xi_1(y) - \xi_n(y) \right)^+}{y - \xi_n(y)} - \frac{ \PP \left[ \bar{M}_{n} \geq y \right]-\PP \left[ \bar{M}_{1} \geq y \right] }{y-\xi_n(y)},
\end{align*}
which implies by induction hypothesis
\begin{align*}
  \frac{\PP \left[ \bar{M}_n \geq y \right]}{y-\xi_n(y)} + \frac{\partial \PP \left[ \bar{M}_{n} \geq y \right]}{\partial y} = 0.  
\end{align*}
This finishes the proof.
\end{proof}

Finally, we argue that $\PP \left[ \bar{M}_n \geq y \right] = K_n(y)$ holds for all $y \geq 0$.

\begin{Proposition}[Law of the Maximum]
\label{prop:Maximizing the Maximum}
Let $n\in \N$ and let Assumption \ref{ass:unicity_minimizers} hold. 
Assume that the embedding property of Theorem \ref{thm:main_result} is valid for the first $n-1$ marginals.
Then for all $y \geq 0$ we have
\begin{align}
\PP \left[ \bar{M}_n \geq y \right] = K_n(y).
%\label{eq:Maximizing the Maximum}
\end{align}
\end{Proposition}
\begin{proof}
The case $n=1$ holds by the Az\'{e}ma-Yor embedding.
Assume by induction hypothesis that 
\begin{align*}
K_i = \PP\left[ \bar{M}_i \geq \cdot \right], \qquad i=1,\dots, n-1.
\end{align*}

In Lemma \ref{lem:ODE_K_n} and \ref{lem:ODE_for_the_Maximum} we derived an ODE for $K_n$ and $\PP\left[ \bar{M}_n \geq \cdot \right]$, respectively, in terms of $K_1,\dots,K_{n-1}$ and $\PP\left[ \bar{M}_1 \geq \cdot \right],\dots,\PP\left[ \bar{M}_{n-1} \geq \cdot \right]$, respectively. These ODEs are valid for a.e. $y \geq 0$.
%Both, \eqref{eq:ODE_K_n} and \eqref{eq:ODE_law_maximum}, constitute an ODE with continuous driver on the interval where $\jmath$ is constant. 
By induction hypothesis both drivers of these ODEs coincide everywhere and hence the claim follows from the boundary conditions
\begin{alignat*}{7}
&K_n(y) && \to 0 \quad &&\text{as $y \to \infty$}, \qquad \qquad && K_n(y) && \to 1 \quad &&\text{as $y \to 0$}, \\
&\PP\left[ \bar{M}_n \geq y \right] && \to 0 \quad &&\text{as $y \to \infty$}, \qquad \qquad && \PP\left[ \bar{M}_n \geq y \right] && \to 1 \quad &&\text{as $y \to 0$},
\end{alignat*}
absolute continuity of $K_n$ and $\PP\left[ \bar{M}_n \geq \cdot \right]$ and the fact that the ODE
\begin{align*}
\Big( \PP\left[ \bar{M}_n \geq y \right] - K_n(y) \Big)' = - \frac{\PP\left[ \bar{M}_n \geq y \right] - K_n(y)}{y-\xi_n(y)}, \qquad \PP\left[ \bar{M}_n \geq 0 \right] - K_n(0) = 0,
\end{align*}
has unique solution given by $0$.
\end{proof}

\subsection{Embedding Property}

In this subsection we prove that the stopping times $\tau_1, \dots, \tau_n$ from Definition \ref{def:iterated_AY} embed the laws $\mu_1,\dots,\mu_n$ if Assumption \ref{ass:unicity_minimizers} is in place. More precisely, given Proposition \ref{prop:Maximizing the Maximum} above and by inductive reasoning, to complete the proof of Theorem \ref{thm:main_result} we only need to show the following:
\begin{Proposition}[Embedding]
\label{prop:embedding}
In the setup of Theorem \ref{thm:main_result} we have
\begin{align}
B_{\tau_n} \sim \mu_n   
\end{align}
and $\left( B_{\tau_n \wedge t} \right)_{t \geq 0}$ is a uniformly integrable martingale.
\end{Proposition}

\begin{proof} 
The case $n=1$ is just the Az\'{e}ma-Yor embedding. By induction hypothesis, assume that the claim holds for all $i \leq n-1$.

We claim that $\xi_n$ ranges continuously over the full support of $\mu_n$. This is because, firstly, we know from Lemma \ref{lem:Monotonicity_of_Stopping_Boundaries} that $\xi_2, \dots, \xi_n$ are continuous. 
Secondly, we have by using $l_{\mu_n} \leq l_{\mu_i}$ that
\begin{align*}
\inf_{\zeta \leq 0} \frac{c^n(\zeta,0)}{0-\zeta} 
\geq \inf_{\zeta \leq 0} \min_{1 \leq i<n} \Bigg\{ \underbrace{\frac{c_n(\zeta)-c_i(\zeta)}{0-\zeta}}_{\geq 0}  + \underbrace{K_i(0)}_{=1} \Bigg\} \wedge \underbrace{\frac{c_n(l_{\mu_n})}{0-l_{\mu_n}}}_{=1} = 1
\end{align*}
which shows that $\xi_n(0)=l_{\mu_n}$. 
Furthermore, by using $r_{\mu_n} \geq r_{\mu_i}$ we have from \eqref{eq:def_barycenter} and \eqref{eq:lower_bound_for_xi_n}
that
\begin{align*}
\xi_n(r_{\mu_n}) = r_{\mu_n}.
\end{align*}

Let $y>0$ be such that $\xi_n$ is differentiable and strictly increasing at $y$, $\xi_n(y)$ is not an atom of neither $\mu_n$ nor $\mu_{\jmath_n(y)}$ and $y$ is not a discontinuity of $\xi_1$. 
Note that for such a $y$ equation \eqref{eq:optimality_xi} holds because of \eqref{eq:constant_xi_n}.
Applying previous results we obtain 
\begin{alignat}{7}
& &&  \quad \quad &&  \PP \left[ M_n \geq \xi_n(y) \right] - \PP\left[ M_{\jmath_n(y)} \geq \xi_n(y) \right] + \PP\left[ \bar{M}_{\jmath_n(y)} \geq y \right] \nonumber \\
& \stackrel{\mathclap{\text{Lemma \ref{lem:Contributions to the Maximum}}}}{=} && \quad && \PP \left[ \bar{M}_n \geq y \right] \nonumber \\
& \stackrel{\mathclap{\text{Prop. \ref{prop:Maximizing the Maximum}}}}{=} && \quad && K_n(y)  \nonumber \\
& \stackrel{\mathclap{\eqref{eq:optimality_xi}}}{=} && \quad && - c'_n(\xi_n(y)) +   c'_{\jmath_n(y)}(\xi_n(y)) + K_{\jmath_n(y)}(y),  \nonumber
\end{alignat}
which implies by induction hypothesis that
\begin{align*}
\PP \left[ M_n \geq \xi_n(y) \right] = - c'_n(\xi_n(y)) = \mu_n([\xi_n(y),\infty)).
%\label{eq:embedding_property_proof}
\end{align*}
We have matched the distribution of $M_n$ to $\mu_n$ at almost all points inside the support.
The embedding property follows.

Now we prove uniform integrability by applying a result from \citet{Azema1980} which states that if
\begin{align}
\lim_{x \to \infty} x\PP\left[ \bar{|B|}_{\tau_n} \geq x \right] = 0
\label{eq:ui_condition_AGY}
\end{align}
then $\left( B_{\tau_n \wedge t} \right)_{t \geq 0}$ is uniformly integrable.

Let us verify \eqref{eq:ui_condition_AGY}. Set $H_x = \inf \left\{ t>0: B_t = x \right\}$. 
We have (here $\xi_i^{-1}$ denotes the left-continuous inverse of $\xi_i$)
\begin{align*} 
\PP\left[ \bar{|B|}_{\tau_n} \geq x \right] &\leq \PP\left[ H_{-x}<H_{\max_{i \leq n}\xi_i^{-1}(-x)} \right] + \PP \left[ \bar{B}_{\tau_n} \geq x  \right] \\
													 &=\frac{\max_{i \leq n}\xi_i^{-1}(-x)}{x + \max_{i \leq n}\xi_i^{-1}(-x)} + K_n(x).
\end{align*}

From the definition of $\xi_n$, cf. \eqref{eq:definition_xi_n}, and the properties of $b_i$, cf. \eqref{eq:def_barycenter} we have
\begin{align*}
0 \leq \max_{i \leq n}\xi_i^{-1}(-x) \leq \max_{i \leq n}b_i(-x) \underset{x \to \infty}{ \longrightarrow } 0
\end{align*}
and hence, recalling the definition of $\mu_n^{\mathrm{HL}}$ in \eqref{eq:def_HL_transform},
\begin{align*} 
\lim_{x \to \infty}x\PP\left[ \bar{|B|}_{\tau_n} \geq x \right] \leq \lim_{x \to \infty} x K_n(x) \leq \lim_{x \to \infty} x \frac{c_n(b_n^{-1}(x))}{x - b^{-1}_n(x)} = \lim_{x \to \infty} x \mu_n^{\mathrm{HL}}([x,\infty)) = 0. 
\end{align*}
This finishes the proof.
\end{proof}

\section{Discussion of Assumption \ref{ass:unicity_minimizers} and Extensions}
\label{sec:Extensions}

In this section we focus on our main technical assumption so far: the condition (ii) in Assumption \ref{ass:unicity_minimizers}. We construct a simple example of probability measures $\mu_1,\mu_2,\mu_3$ which violate the condition and where the stopping boundaries $\xi_1,\xi_2,\xi_3$, obtained via \eqref{eq:definition_xi_n}, fail to embed $(\mu_1,\mu_2,\mu_3)$. It follows that the assumption is not merely technical but does rule out certain type of interdependence between the marginals. If it is not satisfied then it may not be enough to perturb the measures slightly to satisfy it.

We then present an extension of our embedding, in the case $n=3$, which works in all generality. More precisely, we show how to modify the optimisation problem from which $\xi_3$ is determined in order to obtain the embedding property.
The general embedding, as compared to the embedding in the presence of Assumption \ref{ass:unicity_minimizers}(ii), gains an important degree of freedom and becomes less explicit. In consequence it is also much harder to implement in practice, to the point that we do not believe this is worth pursuing for $n>3$. This is also why, as well as for the sake of brevity, we keep the discussion in the section rather formal.

%The example and the reasoning for the general embedding in the case of $n=3$ in this section suggest that Assumption \ref{ass:unicity_minimizers} is central to our previous arguments.
%In addition, a \enquote{practical} sufficient condition seems hard to obtain because of the subtle, multi-dimensional nature of Assumption \ref{ass:unicity_minimizers}.

\subsection{Counterexample for Assumption \ref{ass:unicity_minimizers}(ii)}

In Figure \ref{fig:potentials_couterexample} we define measures via their potentials
\begin{align}
U\mu:\R \to \R, \quad x \mapsto \quad U\mu(x):= -\int_{\R}{ \left| u-x \right| \mu(\dd u) }.
\end{align}
We refer to \citet[Proposition 2.3]{Obloj:04b} for useful properties of $U\mu$.

\begin{figure}[htb]
\begin{center}
\includegraphics[scale=0.85]{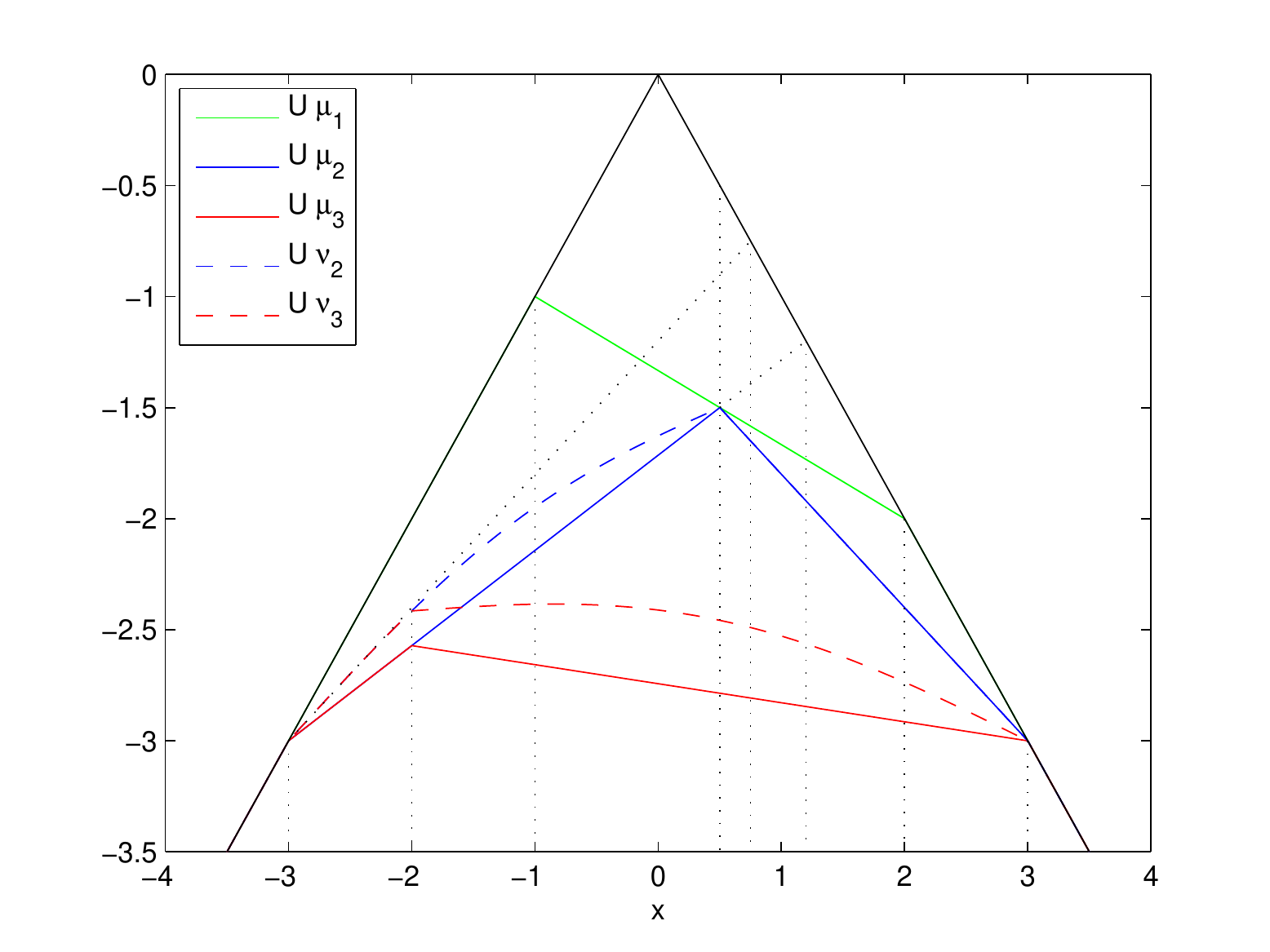}  
\caption{Potentials of $\mu_1,\mu_2$, $\mu_3, \nu_2$ and $\nu_3$.}
\label{fig:potentials_couterexample}
\end{center} 
\end{figure}

The measures with potentials illustrated in Figure \ref{fig:potentials_couterexample} are given as
\begin{alignat}{5}
&\mu_1(\{-1\}) = \frac{2}{3}, \quad &&\mu_1(\{2\}) = \frac{1}{3}, \quad && \label{eq:counterexample_mu_1} \\
&\mu_2(\{-3\}) = \frac{2}{7}, \quad &&\mu_2\left(\left\{\frac{1}{2}\right\}\right) = \frac{18}{35}, \quad &&\mu_2(\{3\}) = \frac{1}{5},  \label{eq:counterexample_mu_2} \\
&\mu_3(\{-3\}) = \frac{2}{7}, \quad &&\mu_3(\{-2\}) = \frac{9}{35}, \quad &&\mu_3(\{3\}) = \frac{16}{35}. \label{eq:counterexample_mu_3}
\end{alignat} 

Observe that the embedding for $\left(\mu_1, \mu_2, \mu_3\right)$ is unique: We write $H_{a,b}$ for the exit time of $[a,b]$ and denote $H_{a,b} \circ \theta_\tau := \inf\left\{ t>\tau:B_t \not \in (a,b) \right\}$. Then the embedding $\left(\tau_1,\tau_2',\tau_3\right)$ can be written as
\begin{align}
\tau_1 = H_{-1,2}, \quad \tau_2' =  H_{-3,\frac{1}{2}} \circ \theta_{\tau_1}\indicator{B_{\tau_1}=-1} + H_{\frac{1}{2},3} \circ \theta_{\tau_1}\indicator{B_{\tau_1}=2}, \quad \tau_3 = H_{-2,3} \circ \theta_{\tau_2}.
\label{eq:unique_stopping_time_counterexample}
\end{align}

As mentioned earlier, our construction yields the same first two stopping boundaries as the method of \citet{Brown98themaximum}. In this case, cf. Figure \ref{fig:stopping_boundaries_couterexample},
\begin{align*}
\xi_1(y):= \begin{cases} 
-1 & \text{if $y \in [0,2)$,}   \\
y & \text{else,}
\end{cases}
\qquad \qquad  
\xi_2(y):= \begin{cases} 
-3 & \text{if $y \in [0,\frac{1}{2})$,}   \\
\frac{1}{2}& \text{if $y \in [\frac{1}{2},3)$,}   \\
y  & \text{else.}
\end{cases}
\end{align*}
This already shows that our embedding fails to embed $\mu_2$. To see this one just has to compare the stopping boundary $\xi_2$ in the Definition of $\tau_2$ with \eqref{eq:unique_stopping_time_counterexample}.
In Section \ref{subsec:general_embedding_n_3} we will recall from \citet{Brown98themaximum} how the stopping time $\tau_2$ has to be modified into $\tau_2'$, giving the stopping time above.

More importantly, the embedding for $\mu_3$ fails because the optimization problem \eqref{eq:definition_xi_n} does not return the third (unique) stopping boundary which is required for the embedding of $(\mu_1,\mu_2,\mu_3)$. Indeed, for sufficiently small $y>\frac{1}{2}$, in the region $\zeta<\min(\xi_1(y),\xi_2(y)) = -1$ we are looking at the minimization of $\zeta \mapsto \frac{c_3(\zeta)}{y-\zeta}$ which is attained by $\xi_3(y)=-3<-2$ since $\mu_3$ has an atom at $-3$. Consequently, there will be a positive probability to hit $-3$ after $\tau_2$. This contradicts \eqref{eq:counterexample_mu_3}. This, together with the correct third boundary $\tilde \xi_3$, is illustrated in Figure \ref{fig:stopping_boundaries_couterexample}.

\begin{figure}[htb]
\begin{center}
\includegraphics[scale=0.70]{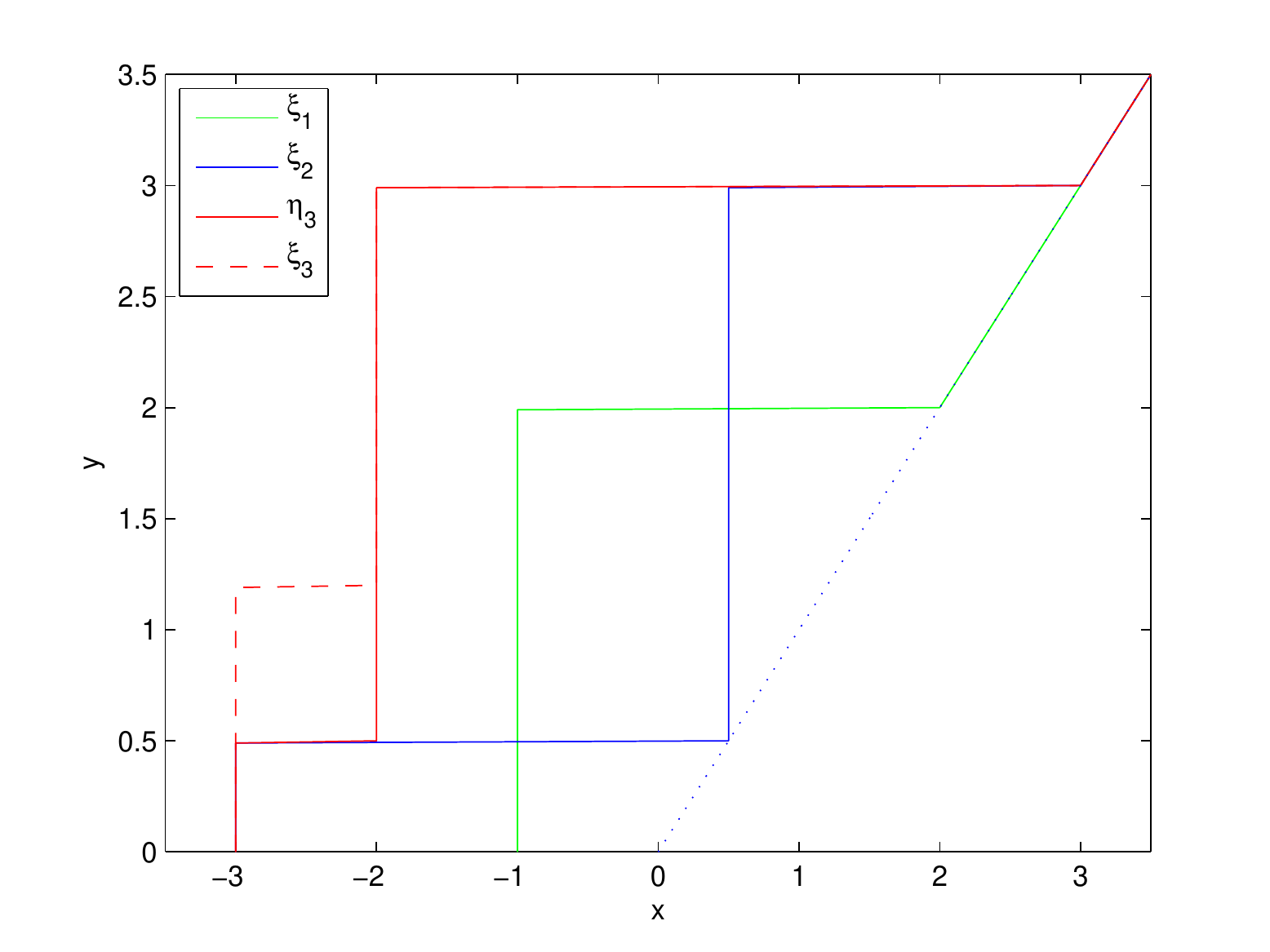}  
\caption{We illustrate the (unique) boundaries $\xi_1,\xi_2,\eta_3$ required for the embedding of $(\mu_1,\mu_2,\mu_3)$ from \eqref{eq:counterexample_mu_1}--\eqref{eq:counterexample_mu_3} and the stopping boundary $\xi_3$ obtained from \eqref{eq:definition_xi_n}. 
In order to ensure the embedding for $\mu_2$, the mass stopped at $\tau_2$ in $-1$ on the event $\{\bar{B}_{\tau_2}\in (1/2,2)\}$ is diffused to $-3$ or to $1/2$ at $\tau_2'$, without affecting the maximum: $\bar{B}_{\tau_2}=\bar{B}_{\tau_2'}$. 
Note that the case $\xi_2(y)=y$, here for $y=1/2$, is possible and required to define the embedding. After $\tau_2'$ we need to define $\tau_3$ which embeds $\mu_3$ which here is implied directly by \eqref{eq:unique_stopping_time_counterexample}. 
In Section \ref{subsec:general_embedding_n_3} we develop arguments which generalise this.}
\label{fig:stopping_boundaries_couterexample}
\end{center} 
\end{figure}

This example does not contradict our main result because Assumption $\ref{ass:unicity_minimizers}$(ii)(a) is not satisfied for $i=2$ and $y=\frac{1}{2}$, where $\zeta=-3$ minimizes the objective function but $c_2\left( \frac{1}{2} \right) = c_1\left( \frac{1}{2} \right)$ holds.
%
%Let $(X_1,X_2)$ be a martingale with marginals $\nu_1$ and $\nu_2$.
%\citet{RePEc:arx:papers:1304.2141} consider optimal lower bounds on $\E{\left| X_1 - X_2 \right|}$ of the form $\int g \dd \nu_1 + \int h \dd \nu_2$. 
%For their construction they require some condition on the difference of the distribution functions of $\nu_1$ and $\nu_2$.
%In our construction we have the additional state variable of the continuously sampled maximum and our construction depends on this variable. 
%Hence, the fact that Assumption \ref{ass:unicity_minimizers} features quantities which relate to this variable should not surprise. 
%
Our counterexample also shows that a \enquote{small perturbation} to $(\mu_1,\mu_2,\mu_3)$ does not remove the problem. Indeed, similar reasoning to the one above holds for measures $(\mu_1,\nu_2,\nu_3)$ defined by their potentials in Figure \ref{fig:potentials_couterexample}.
Assumption \ref{ass:unicity_minimizers} rules out certain type of subtle structure between the marginals and not only some ``isolated" or ''singluar" configurations of measures. 
%In the case of $n=2$ it one can see that Assumption \ref{ass:unicity_minimizers} holds if there is no subinterval on which measures $\mu_1, \mu_2$ have the same mass and mean. This is true in particular if $U\mu_1<U\mu_2$. However for $n\geq 3$ an analogue condition is no longer sufficient as the interaction between all the marginals is important. 

%\input{general_embedding_n_3}

\subsection{Sketch for General Embedding in the Case $n=3$ }
\label{subsec:general_embedding_n_3}

In the example of the measures $(\mu_1,\mu_2,\mu_3)$ from \eqref{eq:counterexample_mu_1}--\eqref{eq:counterexample_mu_3} the (unique) embedding could still be seen as a type of \enquote{iterated Az\'{e}ma-Yor type embedding} although it does not satisfy the relations from Lemma \ref{lem:Key_Properties_of_the_Stopping_Rule}.
Consequently, one might conjecture that a modification of the optimization problem \eqref{eq:definition_xi_n} and a relaxation of Lemma \ref{lem:Key_Properties_of_the_Stopping_Rule} might lead to a generally applicable embedding. We now explain in which sense this is true. Our aim is to outline new ideas and arguments which are needed. The technical details quickly become very involved and lengthy. In the sake of brevity, but also to better illustrate the main points, we restrict ourselves to a formal discussion and the case $n=3$.

In order to understand the problem in more detail, we need to recall from \citet{Brown98themaximum} how the embedding for $\mu_2$ looks like in general. 
It reads
\begin{align}
\tau^{\mathrm{BHR}}_2:= \begin{cases} 
\tau'_2 & \text{if $\xi_2^{-1} \neq \xi_1(\bar{B}_{\tau_1})$ and $\xi_1(\bar{B}_{\tau_1}) < \xi_2(\bar{B}_{\tau_1})$,}   \\
\tau_2 & \text{else,} 
\end{cases}
\label{eq:extensions_BHR_embedding}
\end{align}
where $\tau'_2$ is some stopping time with $\bar{B}_{\tau_2}=\bar{B}_{\tau_2'}$. Its  existence is established by \citet{Brown98themaximum} by showing that the relative parts of the mass which are further diffused have the same mass, mean and are in convex order. 
In general there will be infinitely many such stopping times $\tau'_2$. Although this is not true for $(\mu_1,\mu_2,\mu_3)$ in \eqref{eq:counterexample_mu_1}--\eqref{eq:counterexample_mu_3} because their embedding was unique, it is true for  measures $(\mu_1,\nu_2,\nu_3)$ which are defined via their potentials in Figure \ref{fig:potentials_couterexample}. 

Let $\xi_1$ and $\xi_2$ be defined as in \eqref{eq:definition_xi_n} and let $M_2 = B_{\tau_2}$.
Now our goal is to define an embedding $\tilde{\tau}_3$ for the third marginal on top of the embedding of \citet{Brown98themaximum} in a situation as in $\left(\mu_1, \nu_2, \nu_3 \right)$. We still want to define our iterated Az\'{e}ma-Yor type embedding through a stopping rule based on some stopping boundary $\tilde{\xi}_3$ as a first exit time,
\begin{align}
\tilde{\tau}_3:= \begin{cases} 
\inf\left\{ t \geq \tau^{\mathrm{BHR}}_2 : B_t \leq \tilde{\xi	}_3(\bar{B}_t) \right\} & \text{if $B_{\tau^{\mathrm{BHR}}_2} > \tilde{\xi}_3(\bar{B}_{\tau^{\mathrm{BHR}}_2})$,}   \\
\tau^{\mathrm{BHR}}_2 & \text{else,} 
\end{cases}
\label{eq:extension_tau_3}
\end{align}
and prove that this is a valid embedding of $\mu_3$.
We observe that now the choice of $\tau'_2$ in the definition of $\tau^{\mathrm{BHR}}_2$ may matter  for the subsequent embedding. Similarly as in the embedding of \citet{Brown98themaximum} we expect that this will be only possible if the procedure which produces $\tilde{\xi}_3$ yields a continuous $\tilde{\xi}_3$. Otherwise an additional step, producing a stopping time $\tau_3'\geq \tilde{\tau}_3$ would be required and further complicate the presentation.

With this, a more canonical approach in the context of Lemma \ref{lem:Key_Properties_of_the_Stopping_Rule} is to write
\begin{align}
\label{eq:book_keeping}
\PP \left[ \bar{M}_3 \geq y \right] &= \PP \left[ M_3 \geq \tilde{\xi}_3(y) \right] + \text{\enquote{error-term}},
\end{align}
which we formalize in \eqref{eq:extension_decomposition_of_maximum}.
As it will turn out, this \enquote{error-term} provides a suitable \enquote{book-keeping procedure} to keep track of the masses in the embedding. 
We proceed along the lines of the proof of our main result. For simplicity, we further assume that $\xi_2$ has only one discontinuity, i.e. $\underline{z}:=\xi_2(\underline{y}-)<\xi_2(\underline{y}+):=\overline{z}$ for some $\underline{y} \geq 0$ and we let $\bar{y}:=\xi_1^{-1}(\overline{z})$. 
As explained below, this is not restrictive since our procedure is localised.
If $\bar{y} \leq \underline{y}$ then $\mu_1$ can be \enquote{ignored} and the results of \citet{Brown98themaximum} apply. Hence we assume $\bar{y} > \underline{y}$.

\subsubsection{Redefining $\xi_3$ and $K_3$}

Define the following auxiliary terms,
\begin{alignat}{3}
&F(\zeta,y;\tau'_2) &&:= \indicator{\bar{M}_{1} \geq y}  \big( \zeta - M_2 \big)^+,\label{eq:aux_term} \\
&f^{\mathrm{iAY}}(\zeta,y;\tau'_2) &&:= \EE\left[ F(\zeta,y;\tau'_2) \right].
\end{alignat}
As the notation underlines, these quantities may depend on the additional choice of stopping time $\tau_2'$ between $\tau_2$ and $\tau_3$. Note that for $\zeta \in [\underline{z}, \overline{z}]$ and $y \in [\underline{y}, \bar{y}]$,
\begin{align}
\frac{\partial f^{\mathrm{iAY}}}{\partial \zeta}(\zeta,y;\tau'_2) &= \PP\left[ \bar{M}_1 \geq y, M_2 < \zeta \right], \label{eq:partial_derivative_f_x} %\\
%\frac{\partial^2 f^{\mathrm{iAY}}}{\partial \zeta \partial y}(\zeta,y;\tau'_2) &= -\frac{\PP\left[ \bar{M}_1 \in \dd y, M_2 < \zeta \right]}{\dd y}, \label{eq:partial_derivative_f_x_y}
\end{align}
and 
\begin{align}
\frac{\partial f^{\mathrm{iAY}}}{\partial y}(\zeta,y;\tau'_2) 
&= -\EE\left[ \frac{\indicator{\bar{M}_1 \in \dd y, M_2<\zeta}}{\dd y} \right] \zeta +  \EE\left[ \frac{\indicator{\bar{M}_1 \in \dd y, M_2<\zeta}}{\dd y} M_2 \right] \nonumber \\
&=-\Big(\zeta - \alpha(\zeta,y;\tau'_2)\Big)\frac{\PP\left[ \bar{M}_1 \in \dd y, M_2<\zeta \right]}{\dd y}
\label{eq:partial_derivative_f_y}
\end{align}
where
\begin{align}
\alpha(\zeta,y;\tau'_2) &:= \EE\left[ M_2 \big \vert \bar{M}_1 = y, M_2<\zeta \right], \\
\beta(\zeta,y;\tau'_2)  &:= \EE\left[ M_2 \big \vert \bar{M}_1 = y, M_2 \geq \zeta \right].
\end{align}
With these definitions we have by the properties of $\tau'_2$,
\begin{equation}
\begin{split} 
  \alpha(\zeta,y;\tau'_2)&\frac{\PP\left[ \bar{M}_1 \in \dd y, M_2<\zeta \right]}{\dd y} + \beta(\zeta,y;\tau'_2)\frac{\PP\left[ \bar{M}_1 \in \dd y, M_2 \geq \zeta \right]}{\dd y} \\
= &\ \xi_1(y) \frac{\PP\left[ \bar{M}_1 \in \dd y  \right]}{\dd y}.
\end{split}
\label{eq:relation_alpha_beta}
\end{equation}

We now redefine $\xi_3$ and $K_3$ from \eqref{eq:definition_xi_n} and \eqref{eq:definition_Kn}, respectively, and denote the new definition by $\tilde{\xi}_3$ and $\tilde{K}_3$. 
To this end, introduce the function
\begin{numcases}{\tilde{c}^3(\zeta,y):=}
c_3(\zeta)-f^{\mathrm{iAY}}(\zeta,y;\tau'_2)& \text{if $\underline{z} \leq \zeta \leq \overline{z}, \underline{y} \leq y \leq \bar{y}$, \qquad}
\label{eq:redefintion_tilde_c_3_a} \\
c^3(\zeta,y)    & \text{else.}
\label{eq:redefintion_tilde_c_3_b}
\end{numcases} 

We have that $\tilde{c}^3$ is continuous and $\tilde{c}^3 \leq c^3$.
Using the properties of $\tau'_2$ this can be seen from the following:
%\begin{align}
%\E{ \pp{ M_2 - \overline{z} } \indicator{ \bar{M}_1 \geq y } } = \E{ \pp{ M_1 - \overline{z} } \indicator{ \bar{M}_1 \geq y } }.
%\end{align}
\begin{align}
f^{\mathrm{iAY}}(\zeta,y;\tau'_2) &= \E{ \pp{ \zeta - M_2 } \indicator{ \bar{M}_1 \geq y } } = \E{ \left\{ \pp{ M_2 - \zeta } - (y - \zeta) \right\} \indicator{ \bar{M}_1 \geq y } }   \nonumber \\
&= \E{ \pp{ M_2 - \zeta } \indicator{ \bar{M}_1 \geq y } } - (y - \zeta) K_1(y)   \\
&\geq \begin{cases} c_1(\zeta) - (y-\zeta)K_1(y) &\text{if $\zeta > \xi_1(y)$,}  \\
0 & \text{else. } \end{cases}  \nonumber
\end{align}
for $\zeta \in [\underline{z}, \overline{z}]$ and $y \in [\underline{y}, \bar{y}]$, with equality for $\zeta = \overline{z}$. Continuity at $\zeta = \underline{z}$ holds by the properties of $\tau'_2$. For $y=\bar{y}$ we have $f^{\mathrm{iAY}}(\zeta,\bar{y};\tau'_2) = 0$.
As for continuity at $y=\underline{y}$ it is enough to observe
\begin{align*}
\E{ \pp{ M_2 - \zeta } \indicator{ \bar{M}_1 \geq \underline{y} } } = c_2(\zeta) - (\underline{y}-\zeta)(K_2(\underline{y})-K_1(\underline{y})).
\end{align*}

As before, let 
\begin{align}
\tilde{\xi}_3(y):=\argmin_{\zeta < y}\frac{\tilde{c}^3(\zeta,y)}{y-\zeta}
\end{align}
and 
\begin{align}\label{eq:def_tildeK3}
\tilde{K}_3(y):=\frac{\tilde{c}^3(\tilde{\xi}_3(y),y)}{y-\tilde{\xi}_3(y)}.
\end{align}

It is clear that a discontinuity of $\xi_2$ results in a local perturbation of $c^3$ into $\tilde c^3$ and in consequence of $\xi_3$ into $\tilde \xi_3$. If $\xi_2$ has multiple discontinuities the construction above applies to each of them giving a global definition of $\tilde c^3$. Then $\tilde K_3$ and $\tilde \xi_3$ are defined as above.

%We stress that now the quantities $\tilde{\xi}_3$ and $\tilde{K}_3$ may depend on $\tau'_2$.

\subsubsection{Law of the Maximum}

In the following we assume that $\zeta \in [\underline{z}, \overline{z}]$ and $y \in [\underline{y}, \bar{y}]$. Otherwise $\tilde c^3 = c^3$ and the arguments from Sections \ref{sec:Main Assumption and Definitions} and \ref{sec:iAY} apply.
We have $\tilde{\xi}_3(y) < \xi_2(y)$ and $\bar{M}_1 = \bar{M}_2$ on $\{ \bar{M}_1 \in [ \underline{y}, \bar{y} ] \}$.
 
Note the obvious decomposition 
\begin{align*}
\PP \left[ \bar{M}_3 \geq y \right] = \PP \left[ \bar{M}_{1} < y, \bar{M}_3 \geq y \right] + \PP \left[ \bar{M}_{1} \geq y \right].
%\label{eq:extension_trivial_decomposition}
\end{align*}
 
We compute by similar excursion theoretical arguments as in the proof of Lemma \ref{lem:ODE_for_the_Maximum}, 
\begin{equation}
\begin{split}
&\restr{\frac{\partial \PP \left[ \bar{M}_{1} <y, \bar{M}_3 \geq m \right]}{\partial y}}{m=y} =: p(\tilde{\xi}_3(y),y;\tau'_2) \\ 
=&\frac{ \PP\left[ M_2>\tilde{\xi}_3(y), \bar{M}_1 \in \dd y \right]}{\dd y} \frac{\beta(\tilde{\xi}_3(y),y;\tau'_2) - \tilde{\xi}_3(y)}{y-\tilde{\xi}_3(y)}
%\cdot \PP \left[ H_y \circ \theta_{\tau'_2} < H_{\tilde{\xi}_3(y)} \circ \theta_{\tau'_2} \big\vert M_2>\tilde{\xi}_3(y), \bar{M}_1 = y \right]   
\end{split}
\label{eq:definition_p_probability} 
\end{equation}
%where as above $H_x:=\inf\left\{ t \geq 0:B_t=x \right\}$ and

In analogy to \eqref{eq:right_derivative_aux}, and because $\tilde{\xi}_3(y)<\xi_2(y)$,
\begin{align*}
\restr{\frac{\partial \PP \left[ \bar{M}_{1} < m, \bar{M}_3 \geq y \right]}{\partial y}}{m=y} = -\frac{\PP \left[ \bar{M}_3 \geq y \right] - \PP\left[ \bar{M}_1 \geq y \right] }{y-\tilde{\xi}_3(y)}.
\end{align*}

Hence, combining the above
\begin{align}
\frac{\partial}{\partial y}\PP \left[ \bar{M}_3 \geq y \right] 
&= p(\tilde{\xi}_3(y),y;\tau'_2)  -\frac{\PP \left[ \bar{M}_3 \geq y \right] - \PP\left[ \bar{M}_1 \geq y \right] }{y-\tilde{\xi}_3(y)}   + \frac{\partial \PP\left[ \bar{M}_1 \geq y \right]}{\partial y}  \nonumber \\
&\stackrel{\mathclap{\eqref{eq:ODE_law_maximum}}}{=} -\frac{\PP \left[ \bar{M}_3 \geq y \right] }{y-\tilde{\xi}_3(y)} -  \frac{\tilde{\xi}_3(y)-\xi_1(y)}{y-\tilde{\xi}_3(y)} \frac{\partial \PP\left[ \bar{M}_1 \geq y \right]}{\partial y} +   p(\tilde{\xi}_3(y),y;\tau'_2).
\label{eq:extension_M_3_derivative}
\end{align}

In the redefined domain the first order condition for optimality of $\tilde{\xi}_3(y)$ reads
\begin{align}
\tilde{K}_3(y) +  c'_3(\tilde{\xi}_3(y)) - \frac{\partial f^{\mathrm{iAY}}}{\partial \zeta}(\tilde{\xi}_3(y),y;\tau'_2)  = 0.
\label{eq:optimality_tilde_xi}
\end{align}  

By similar calculations as in \eqref{eq:ODE_K_N_proof} below we have
\begin{alignat}{3}
\tilde{K}'_3(y) \stackrel{\mathclap{\eqref{eq:optimality_tilde_xi}}}{=} &-&& \frac{\tilde{K}_3(y)}{y-\tilde{\xi}_3(y)} - \frac{\frac{\partial f^{\mathrm{iAY}}}{\partial y}(\tilde{\xi}_3(y),y;\tau'_2)}{y-\tilde{\xi}_3(y)} \nonumber \\
				\stackrel{\mathclap{\eqref{eq:partial_derivative_f_y}}}{=} & - && \frac{\tilde{K}_3(y)}{y-\tilde{\xi}_3(y)} + \frac{ \tilde{\xi}_3(y) - \alpha(\tilde{\xi}_3(y),y) }{y-\tilde{\xi}_3(y)} \frac{\PP\left[ \bar{M}_1 \in \dd y, M_2<\tilde{\xi}_3(y) \right]}{\dd y} \nonumber \\
				\stackrel{\mathclap{\eqref{eq:relation_alpha_beta}}}{=} &-&& \frac{\tilde{K}_3(y)}{y-\tilde{\xi}_3(y)} + \frac{ \tilde{\xi}_3(y) - \xi_1(y) }{y-\tilde{\xi}_3(y)}\frac{\PP\left[ \bar{M}_1 \in \dd y  \right]}{\dd y} \nonumber \\
				&+&& \frac{ \beta(\tilde{\xi}_3(y), y) - \tilde{\xi}_3(y) }{y-\tilde{\xi}_3(y)} \frac{\PP\left[ \bar{M}_1 \in \dd y, M_2 \geq \tilde{\xi}_3(y) \right]}{\dd y} \nonumber \\
				\stackrel{\mathclap{\eqref{eq:definition_p_probability}}}{=} &-&& \frac{\tilde{K}_3(y)}{y-\tilde{\xi}_3(y)} - \frac{ \tilde{\xi}_3(y) - \xi_1(y) }{y-\tilde{\xi}_3(y)} \frac{\partial \PP\left[ \bar{M}_1 \geq y  \right]}{\partial y} + p(\tilde{\xi}_3(y),y;\tau'_2). 
				\label{eq:extension_ODE_K_3}
\end{alignat}

Consequently, by comparing \eqref{eq:extension_M_3_derivative} and \eqref{eq:extension_ODE_K_3}, and in conjunction with Proposition \ref{prop:Maximizing the Maximum}, we obtain
\begin{align}
\tilde{K}_3(y) = \PP\left[ \bar{M}_3 \geq y \right], \qquad \text{for all $y \geq 0$.}
\label{eq:extension_maximum_identification}
\end{align}

\subsubsection{Embedding Property}

After having found the distribution of the maximum, the final step is to prove the embedding property. 
To achieve this we will need that $\tilde{\xi}_3$ is non-decreasing.% which we argue by using the (formal) first and second order conditions for optimality of $\tilde{\xi}_3(y)$.
  
Recall the first order condition of optimality of $\tilde{\xi}_3$ in \eqref{eq:optimality_tilde_xi}  
and then the second order condition for optimality of $\tilde{\xi}_3(y)$ reads
\begin{align}
c''_3(\tilde{\xi}_3(y)) - \frac{\partial^2 f^{\mathrm{iAY}}}{\partial \zeta^2}(\tilde{\xi}_3(y),y;\tau'_2) \geq 0.
\label{eq:seond_order_condition_optimality}
\end{align}

Now, differentiating \eqref{eq:optimality_tilde_xi} in $y$ yields
\begin{align*}
\tilde{K}'_3(y) + c''_3(\tilde{\xi}_3(y)) \tilde{\xi}'_3(y) - \frac{\partial^2 f^{\mathrm{iAY}}}{\partial \zeta^2}(\tilde{\xi}_3(y),y;\tau'_2) \tilde{\xi}'_3(y) - \frac{\partial^2 f^{\mathrm{iAY}}}{\partial \zeta \partial y}(\tilde{\xi}_3(y),y;\tau'_2) = 0 
\end{align*}
or equivalently,
\begin{align*}
 \tilde{\xi}'_3(y)  \underbrace{ \left( c''_3(\tilde{\xi}_3(y))  - \frac{\partial^2 f^{\mathrm{iAY}}}{\partial \zeta^2}(\tilde{\xi}_3(y),y;\tau'_2) \right) }_{\geq 0 \text{ by \eqref{eq:seond_order_condition_optimality}}}   = - \tilde{K}'_3(y) + \frac{\partial^2 f^{\mathrm{iAY}}}{\partial \zeta \partial y}(\tilde{\xi}_3(y),y;\tau'_2)
\end{align*}

In order to formally infer
\begin{align*}
\tilde{\xi}'_3(y) \geq 0
\end{align*}
we require
\begin{align}
- \tilde{K}'_3(y) + \frac{\partial^2 f^{\mathrm{iAY}}}{\partial \zeta \partial y}(\tilde{\xi}_3(y),y;\tau'_2) \geq 0.
\label{eq:condition_monotonicity_tilde_xi}
\end{align}

Direct computation shows that
\begin{align*}
\frac{\partial^2 f^{\mathrm{iAY}}}{\partial \zeta \partial y}(\zeta,y;\tau'_2) &= -\frac{\PP\left[ \bar{M}_1 \in \dd y, M_2 < \zeta \right]}{\dd y} %\label{eq:partial_derivative_f_x_y}
\end{align*} 
and by \eqref{eq:extension_maximum_identification},
\begin{align*}
- \tilde{K}'_3(y) = \frac{\Prob{ \bar{M}_3  \in \dd y}}{\dd y}
\end{align*}
which implies \eqref{eq:condition_monotonicity_tilde_xi} and hence that $\tilde \xi_3$ is non-decreasing.

By definition of the embedding in \eqref{eq:extension_tau_3}, and since $\tilde{\xi}_3$ is non-decreasing, we have
\begin{alignat}{3}
\PP \left[ \bar{M}_3 \geq y \right] &= \hspace{3mm} &&\PP \left[ M_3 \geq \tilde{\xi}_3(y) \right] + \PP\left[ \bar{M}_3 \geq y, M_3 < \tilde{\xi}_3(y) \right] \nonumber \\
									 &= &&\PP \left[ M_3 \geq \tilde{\xi}_3(y) \right] + \PP\left[ \bar{M}_1 \geq y, M_2 < \tilde{\xi}_3(y) \right] \nonumber \\
%									 &= &&\PP \left[ M_3 \geq \tilde{\xi}_3(y) \right] %+ \PP\left[ \bar{M}_1 \geq y, M_2 < \tilde{\xi}_3(y) \right] \nonumber \\
&\stackrel{\mathclap{\eqref{eq:partial_derivative_f_x}}}{=} &&\PP \left[ M_3 \geq \tilde{\xi}_3(y) \right] + \frac{\partial f^{\mathrm{iAY}}}{\partial \zeta}(\tilde{\xi}_3(y),y;\tau'_2).
									 \label{eq:extension_decomposition_of_maximum}
\end{alignat}
and then, by \eqref{eq:extension_maximum_identification}, \eqref{eq:optimality_tilde_xi} and \eqref{eq:extension_decomposition_of_maximum},
\begin{align*}
-c'_3(\tilde{\xi}_3(y)) = \PP\left[ M_3 \geq \tilde{\xi}_3(y) \right]
%\label{eq:extension_embedding}
\end{align*}
which is the desired embedding property.
%If $\tilde{\xi}_3$ happens to be discontinuous, then we cannot guarantee the embedding property for $\mu_3$. However, we expect that similar arguments as those in \citet[Section 3.5]{Brown98themaximum} will yield a stopping time $\tau'_3$ which does the job for the jump-intervals of $\tilde{\xi}_3$.
%We do not pursue these matters here. 

The above construction hinged on the appropriate choice of the auxiliary term $F$ in \eqref{eq:aux_term} whose expectation, as follows from  \eqref{eq:redefintion_tilde_c_3_a}, \eqref{eq:def_tildeK3} and \eqref{eq:extension_maximum_identification}, allows for the error book keeping, as suggested in \eqref{eq:book_keeping}. We identified the correct $F$ by analysing the ``error terms" which cause strict inequality for $(B_u:u\leq \tau_2')$ in the pathwise inequality (4.1) of \citet{Touzi_maxmax}. This is natural since this inequality is used to prove optimality of our embedding. It gives an upper bound but fails to be sharp if condition (ii) in Assumption \ref{ass:unicity_minimizers} does not hold. In order to recover a sharp bound one has to look at the error terms causing strict inequality when Assumption \ref{ass:unicity_minimizers} fails. The same principle applies for $n>3$. However then interactions between discontinuities of boundaries $\xi_2,\tilde \xi_3$ etc come into play and the relevant terms become very involved. The construction would become increasingly technical and implicit and we decided to stop at this point.

\appendix

\section{Appendix: Proof of Lemma \ref{lem:ODE_K_n}}
\label{sec:appendix_1}

In order to prove Lemma \ref{lem:ODE_K_n} we require to prove, inductively, several auxiliary results along the way.
We now state and prove a Lemma which contains the statement of Lemma \ref{lem:ODE_K_n}.

\begin{Lemma}
%\label{lem:ODE_K_n}
Let $n\in \N$ and let Assumption \ref{ass:unicity_minimizers} hold. 
Then 
\begin{align}
y \mapsto K_n(y) \qquad \text{is absolutely continuous and non-increasing.}
\end{align}

If we assume in addition that the embedding property of Theorem \ref{thm:main_result} is valid for the first $n-1$ marginals then for almost all $y \geq 0$ we have: 

If $\xi_n(y)<y$ then
\begin{align}
K'_n(y)  + \frac{K_n(y)}{y-\xi_n(y)} = K'_{\jmath_n(y)}(y) + \frac{K_{\jmath_n(y)}(y)}{y-\xi_n(y)}
\label{eq:app_ODE_K_n}
\end{align}
where $K'_j$ denotes the derivative of $K_j$ which exists for almost all $y \geq 0$ and $j=1,\dots,n$.

If $\xi_n(y)=y$ then
\begin{align}
K_n(y+) = K_{\jmath_n(y)}(y+).
\label{eq:app_ODE_K_n_2}
\end{align}

For $x>0$ the mapping 
\begin{align}
c^n: (x,\infty) \to \R, \quad y \mapsto c^n(x,y) 
\end{align}
is locally Lipschitz continuous, non-decreasing and for almost all $y>0$
\begin{align}
\restr{\frac{\partial c^n}{\partial y}(x,y)}{x=\xi_n(y)} = K_{\jmath_n(y)}(y) + (y-\xi_n(y)) K'_{\jmath_n(y)}(y).
\label{eq:ODE_c_n}
\end{align}

The mapping $c^n(\cdot,y)$ is locally Lipschitz continuous and if it is differentiable at $\xi_n(y)$ and $\xi'_n(y)>0$ then for almost all $y \geq 0$  
\begin{align}
K_n(y) + c'_n(\xi_n(y)) - c'_{j}(\xi_n(y)) - K_{j}(y) = 0 % \quad \text{for $j=\imath_n(\xi_n(y)+;y)$.}
\label{eq:optimality_xi}
\end{align}
for $j=\jmath_n(y)$ and $j$ such that $n>j>\jmath_n(y)$ and $\xi_n(y) = \xi_j(y)$.
In the case of non-smoothness of $c^n(\cdot,y)$ at $\xi_n(y)$ we have
\begin{align}
\xi'_n(y)=0
\label{eq:constant_xi_n}
\end{align}
for $y$ such that the slope of the supporting tangent to $c^n(\cdot,y)$ at $\xi_n(y)$ which crosses the $x$-axis at $y$ does not equal the right-derivative of $c^n(\cdot,y)$ at $\xi_n(y)$.
\end{Lemma}
 
\begin{proof}
We prove the claim by induction over $n$. The induction basis $n=1$ holds by definition and Lemma 2.6 of \citet{Brown98themaximum}. 

Now assume that the claim holds for all $i=1,\dots,n-1$.

\textit{Induction step for $c^n$.}
We have
\begin{equation}
\begin{split}
 c^n(x,y+\delta) - c^n(x,y) = &- \left[ c_{\imath_n(x;y+\delta)}(x) - (y+\delta-x)K_{\imath_n(x;y+\delta)}(y+\delta) \right] \\
 							   &+ \left[ c_{\imath_n(x;y)}(x) - (y-x)K_{\imath_n(x;y)}(y) \right].
\end{split}
\label{eq:local_Lipschitz_c_n}
\end{equation}

Firstly, consider the case when there exists a $\delta'>0$ such that for all $|\delta|<\delta'$ we have $\imath_n(x;y) = \imath_n(x;y+\delta)$.
Equation \eqref{eq:local_Lipschitz_c_n} simplifies and we have
\begin{alignat}{3}
 & &&c^n(x,y+\delta) - c^n(x,y) =  (y+\delta-x)K_{\imath_n(x;y)}(y+\delta) - (y-x)K_{\imath_n(x;y)}(y) \nonumber \\
=& && (y+\delta-\xi_{\imath_n(x;y)}(y))K_{\imath_n(x;y)}(y+\delta) - (y-\xi_{\imath_n(x;y)}(y)) K_{\imath_n(x;y)}(y) \nonumber \\
 &+&& (x-\xi_{\imath_n(x;y)}(y)) \left[ K_{\imath_n(x;y)}(y) - K_{\imath_n(x;y)}(y+\delta) \right] \nonumber  \\
 \stackrel{\mathclap{\eqref{eq:definition_xi_n_tangent_interpretation}}}{\leq} & && c^{\imath_n(x;y)}(\xi_{\imath_n(x;y)}(y),y+\delta) - c^{\imath_n(x;y)}(\xi_{\imath_n(x;y)}(y),y) \nonumber   \\ &+&& (y-\xi_{\imath_n(x;y)}(y)) \left[ K_{\imath_n(x;y)}(y) - K_{\imath_n(x;y)}(y+\delta) \right]  \nonumber\\
\leq & && \max_{i<n} \Big\{ c^{i}(\xi_{i}(y),y+\delta) - c^{i}(\xi_{i}(y),y)    + (y-\xi_{i}(y)) \left[ K_{i}(y) - K_{i}(y+\delta) \right] \Big\} \nonumber \\
\leq & && \mathrm{const}(y) \cdot |\delta|
\label{eq:ineq_delta_pos}
\end{alignat} 
by induction hypothesis and where $\mathrm{const}(y)$ denotes a constant depending on $y$.
A similar computation shows for $|\delta|$ small enough
\begin{alignat}{3}
& &&c^n(x,y) - c^n(x,y+\delta) \nonumber \\
 \stackrel{\mathclap{\eqref{eq:definition_xi_n_tangent_interpretation}}}{\leq} & && c^{\imath_n(x;y)}(\xi_{\imath_n(x;y)}(y+\delta),y) - c^{\imath_n(x;y)}(\xi_{\imath_n(x;y)}(y+\delta),y+\delta)  \nonumber \\ 
 &+&& (x-\xi_{\imath_n(x;y)}(y+\delta)) \left[ K_{\imath_n(x;y)}(y+\delta) - K_{\imath_n(x;y)}(y) \right] \nonumber \\
= & && (y - \xi_{\imath_n(x;y)}(y+\delta) ) \left[ K_{\imath_{\imath_n(x;y)}(x;y)}(y) - K_{\imath_{\imath_n(x;y)}(x;y)}(y+\delta) \right] - \delta K_{\imath_{\imath_n(x;y)}(x;y)}(y+\delta) \nonumber \\ 
 &+&& (x-\xi_{\imath_n(x;y)}(y+\delta)) \left[ K_{\imath_n(x;y)}(y+\delta) - K_{\imath_n(x;y)}(y) \right] \nonumber \\
\leq & && \begin{cases} 
\mathrm{const}(y) \cdot |\delta| & \text{if $\delta<0$,} \\
0 & \text{if $\delta \geq 0$,}
\end{cases} 
\label{eq:monotonicity_c_n_special_case} 
\end{alignat}
again by induction hypothesis and continuity of $\xi_i$ which indeed allows the constant to be chosen independently of $\delta$.
Monotonicity of $c^n(x,\cdot)$ follows.
Equation \eqref{eq:monotonicity_c_n_special_case} together with \eqref{eq:ineq_delta_pos} imply the local Lipschitz continuity. 
Plugging $x=\xi_n(y)$ into \eqref{eq:local_Lipschitz_c_n}, direct computation shows that  \eqref{eq:ODE_c_n} holds.

Secondly, consider the case when $\imath_n(x;\cdot)$ jumps at $y$. 
Recall \eqref{eq:right_continuity_imath}. 
Note that this is only possible when $x$ satisfies 
\begin{align}
x=\xi_{\imath_n(x;y-\delta)}(y)  \qquad \text{for $\delta>0$ small enough,}
\label{eq:jump_point_x}
\end{align}
i.e. when $x=\xi_k(y)$ for the index $k=\imath_n(x;y-\delta)>\jmath_n(y)$.
By \eqref{eq:right_continuity_imath} there exists a $\delta'>0$ such that $\imath_n(x;y+\delta) = \imath_n(x;y)$ for all $0 \leq \delta < \delta'$. 
Hence, for $\delta>0$ small enough, $|c^n(x,y+\delta)-c^n(x,y)|$ has the same upper bound as in the first case. Monotonicity of $c^n(x,\cdot)$ follows.

Furthermore, for $\delta>0$ we have for the $x$ from \eqref{eq:jump_point_x} that $\imath_n(x;y-\delta)>\imath_n(x;y)$ holds. For notational simplicity we only consider the case
\begin{align*}
 \imath_{\imath_n(x;y-\delta)}(x;y-\delta) = \imath_n(x;y).
\end{align*}  
The general case follows by the same arguments.
We deduce from \eqref{eq:local_Lipschitz_c_n} the following two equations,
\begin{align}
 c^n(x,y-\delta) - c^n(x,y) & \stackrel{\mathclap{\eqref{eq:definition_xi_n_tangent_interpretation}}}{\leq}  (y-\delta-x) K_{\imath_{\imath_n(x;y-\delta)}(x;y-\delta)}(y-\delta) - (y-x) K_{\imath_n(x;y)}(y) \nonumber \\
 							 &=(y-\delta-x) K_{\imath_n(x;y)}(y-\delta) - (y-x) K_{\imath_n(x;y)}(y)
\label{eq:ineq_delta_neg}
\end{align}
and  
\begin{align*}
 c^n(x,y) - c^n(x,y-\delta) \stackrel{\eqref{eq:jump_point_x}}{=}  -(y-\delta-x) K_{\imath_n(x;y-\delta)}(y-\delta) + (y-x) K_{\imath_n(x;y-\delta)}(y).  
\end{align*}
Now the local Lipschitz continuity of $c^n(x,\cdot)$ follow from the above two equations by repeating the arguments from the first case.

We prove \eqref{eq:ODE_c_n} by computing the required right- and left-derivative of $c^n(x,\cdot)$ at $x=\xi_n(y)$. 
The right-derivative is simply, using \eqref{eq:right_continuity_imath} and \eqref{eq:local_Lipschitz_c_n},
\begin{align}
K_{\jmath_n(y)}(y) + (y-\xi_n(y)) K'_{\jmath_n(y)}(y)
\end{align}
and the left-derivative is, writing $k=\imath_n(\xi_n(y);y-) > \imath_n(\xi_n(y);y)=\jmath_n(y)$,
\begin{alignat}{3}
& \hspace{7mm} &&\lim_{\delta \uparrow 0}  \frac{1}{\delta}\Big( -c_k(\xi_n(y)) + ( y+\delta - \xi_n(y) ) K_k(y+\delta)  \nonumber \\
& &&\hspace{13mm} + c_{\jmath_n(y)}(\xi_n(y)) - (y-\xi_n(y)) K_{\jmath_n(y)}(y) \Big) \nonumber \\
\stackrel{\mathclap{\xi_n(y)=\xi_k(y)}}{=}& &&\lim_{\delta \uparrow 0} \frac{1}{\delta} \Big( ( y+\delta - \xi_n(y) ) K_k(y+\delta) - ( y - \xi_n(y) ) K_k(y) \Big) \nonumber \\
=& && K_{k}(y) + (y-\xi_n(y)) K'_{k}(y) \stackrel{\eqref{eq:app_ODE_K_n}}{=} K_{\jmath_n(y)}(y) + (y-\xi_n(y)) K'_{\jmath_n(y)}(y)
\end{alignat}
by induction hypothesis. So the two coincide for almost all $y > 0$.

\textit{Induction step for $K_n$.} 
A straightforward computation shows that the mapping $y \mapsto \frac{c^n(x,y)}{y-x}$
is non-increasing and hence for $\delta>0$
\begin{align*}
K_n(y+\delta) = \inf_{\zeta \leq y+\delta}\frac{c^n(\zeta,y+\delta)}{y+\delta-\zeta} \leq \inf_{\zeta \leq y}\frac{c^n(\zeta,y+\delta)}{y+\delta-\zeta} \leq \inf_{\zeta \leq y}\frac{c^n(\zeta,y)}{y-\zeta} = K_n(y)
\end{align*}
proving that $K_n$ is non-increasing.

Using again that $c^n(x,\cdot)$ is non-decreasing and that $\xi_n$ is continuous, local Lipschitz continuity of $K_n$ now follows from
\begin{align*}
K_n(y) \leq \frac{c^n(\xi_n(y+\delta),y)}{y-\xi_n(y+\delta)} \leq   \frac{c^n(\xi_n(y+\delta),y+\delta)}{y-\xi_n(y+\delta)} = K_n(y+\delta)\left( 1+\frac{\delta}{y-\xi_n(y+\delta)}\right)
\end{align*}
if $\xi_n(y)<y$ and if $\xi_n(y)=y$, recalling \eqref{eq:extension_c_n}, we have
\begin{alignat*}{3}
K_n(y+) =& \qquad && \inf_{\mathclap{\zeta \leq (y+)}}\hspace{5mm}\left\{\frac{c_n(\zeta) - c_{\imath_n(\zeta;y+)}(\zeta)}{(y+)-\zeta}+K_{\imath_n(\zeta;y+)}(y) \right\} \\
\stackrel{ \mathclap{\text{Lemma \ref{lem:Monotonicity_of_Stopping_Boundaries}}} }{=}& && \inf_{\mathclap{y \leq \zeta \leq (y+)}} \hspace{5mm}\left\{\frac{c_n(\zeta) - c_{\imath_n(\zeta;y+)}(\zeta)}{(y+)-\zeta}+K_{\imath_n(\zeta;y+)}(y+) \right\} = K_{\imath_n(y;y)}(y+)  \\
=& &&  K_{\jmath_n(y)}(y+),
\end{alignat*}
and local Lipschitz continuity of $K_n$ follows by induction hypothesis.
Equation \eqref{eq:app_ODE_K_n_2} is also proven.

%%%%%%%%%%%%%%%%%%%%%%%%%%%%%%%%%%%%%%%%%%%
%%%%%%%%%%%%%%%%%%%%%%%%%%%%%%%%%%%%%%%%%%%
%%%%%%%%%%%%%%%%%%%%%%%%%%%%%%%%%%%%%%%%%%%
\begin{comment}
A direct computation proves the left-limit version of \eqref{eq:optimality_xi} for $j=\jmath_n(y)$. Now fix some $n>l>\jmath_n(y)$ such that $\xi_n(y) = \xi_l(y)$. 
From this we conclude by induction hypothesis and the fact that $\jmath_n(y) = \jmath_l(y)$
\begin{align*}
0 = K_n(y) + c'_n(\xi_n(y) -) - c'_{l}(\xi_n(y) -) + \underbrace{c'_{l}(\xi_l(y) -) - c'_{\jmath_l(y)}(\xi_l(y) -) K_{\jmath_l(y)}(y) }_{=K_l(y)}  
\end{align*}
and \eqref{eq:optimality_xi} follows for $j$ such that $n>j>\jmath_n(y)$ and $\xi_n(y) = \xi_j(y)$.
\end{comment}

Local Lipschitz continuity of $c^n(\cdot,y)$ follows from the properties of $\imath_n$, cf.  \eqref{eq:left_continuity_imath}, the fact that the functions $c_i, i=1,\dots,n$, are locally Lipschitz and a similar expansion of terms in the case when $\xi_n(y)=\xi_i(y)$ for some $i<n$ 
%Let $k=\imath_n(\xi_n(y)+;y)$. This also shows that for any $y\geq 0$ kinks of $c^n(\cdot,y)$ come from kinks in $c_i,i=1,\dots,n$.

In order to prove \eqref{eq:optimality_xi} we first exclude all $y \geq 0$ such that $\xi_n(y)$ is an atom of $c_1,\dots,c_n$ and $\xi'_n(y)>0$. Amongst all $y \in \left\{ \xi'_n > 0 \right\}$ this is a null-set.
By assumption $c^n(\cdot,y)$ is differentiable at $\xi_n(y)$.
Recalling the equations \eqref{eq:continuity_1}--\eqref{eq:continuity_3}, a direct computation proves \eqref{eq:optimality_xi} for $k=\imath_n(\xi_n(y)+;y)$. 
Now we want to apply the induction hypothesis to $c^k$. By choice of $y$ we have that $c_k$ is differentiable at $\xi_n(y)= \xi_k(y)$, i.e. $\mu_k$ does not have an atom at $\xi_n(y)$. Hence, by the assumption that the embedding for the first $n-1$ marginals is valid we cannot have $\xi'_k(y)=0$ (except on a null-set because otherwise the embedding would fail). 
By \eqref{eq:constant_xi_n}, $c^{k}(\cdot,y)$ therefore has to be differentiable at $\xi_n(y)=\xi_k(y)$.
This shows that we can indeed apply \eqref{eq:optimality_xi} to deduce for $j=\jmath_n(y)$ and $j$ such that $n>j>\jmath_n(y)$ and $\xi_n(y) = \xi_k(y) = \xi_j(y)$ the following equation,
\begin{align*}
0&=K_n(y)+c'_n(\xi_n(y))-c'_k(\xi_n(y)) + K_k(y) \\
 &=K_n(y)+c'_n(\xi_n(y))-c'_k(\xi_n(y)) + c'_k(\xi_n(y)) - c'_j(\xi_n(y)) + K_j(y) \\
 &=K_n(y)+c'_n(\xi_n(y)) - c'_j(\xi_n(y))+ K_j(y).
\end{align*}
Equation \eqref{eq:optimality_xi} is proven.

For later use we note the equation
\begin{align}
\frac{c_n(\xi_n(y)) - c_{\jmath_n(y)}(\xi_n(y))}{y-\xi_n(y)} + K_{\jmath_n(y)}(y) =  K_n(y) = \frac{c_n(\xi_n(y)) - c_k(\xi_n(y))}{y-\xi_n(y)} + K_k(y)  
\label{eq:relations_boundary_crossing}
\end{align}
for $k$ such that $n>k>\jmath_n(y)$ and $\xi_n(y) = \xi_k(y)$.

Finally, we prove the claimed ODE for $K_n$  in the case $\xi_n(y)<y$. 
For almost all $y \geq 0$ we have 
\begin{align*}
K'_n(y) &= \lim_{\delta \to 0} \frac{1}{\delta} \left[ \frac{c^n(\xi_n(y+\delta),y+\delta)}{y+\delta-\xi_n(y+\delta)} - \frac{c^n(\xi_n(y),y)}{y-\xi_n(y)} \right]  \\
		&= \lim_{\delta \to 0} \frac{1}{\delta} \left[ \left( \frac{1}{y+\delta-\xi_n(y+\delta)} - \frac{1}{y-\xi_n(y)} \right) c^n(\xi_n(y+\delta),y+\delta) \right. \\
		&  \hspace{18mm} \left. + \frac{c^n(\xi_n(y+\delta),y+\delta) - c^n(\xi_n(y),y)}{y-\xi_n(y)} \right] \\
		&= \frac{\xi'_n(y)-1}{y-\xi_n(y)} K_n(y) + \frac{1}{y-\xi_n(y)}\left( \lim_{\delta \to 0} \frac{c^n(\xi_n(y+\delta),y+\delta) - c^n(\xi_n(y),y) }{\delta} \right).
\end{align*}

The main technical difficulty comes from the possibility that $\xi_n(y) = \xi_k(y)$ for some $k<n$. We present the arguments for this case and leave the other (much easier) case, to the reader.

By assumption the last limit exists and hence we can compute it using some \enquote{convenient} sequence $\delta_m \downarrow 0$ where $\delta_m$ is such that $\jmath_n(y+\delta_m) = l$ for all $m \in \N$. 
Note that by continuity of $\xi_1,\dots,\xi_n$ at $y$ we have that either $l=\jmath_n(y)$ or $l$ is such that $\xi_l(y) = \xi_n(y)$. This will enable us apply \eqref{eq:relations_boundary_crossing}.
Recall \eqref{eq:right_continuity_imath}. For $\delta_m$ small enough such that $\imath_n(\xi_n(y);y+\delta_m) = \jmath_n(y)$ we obtain
\begin{alignat*}{3}
 & && c^n(\xi_n(y+\delta_m),y+\delta_m) - c^n(\xi_n(y),y+\delta_m) \\ 
=& && c_n(\xi_n(y+\delta_m)) - c_l(\xi_n(y+\delta_m)) + (y+\delta_m - \xi_n(y+\delta_m)) K_l(y+\delta_m) \\
 &-&& c_n(\xi_n(y)) + c_{\jmath_n(y)}(\xi_n(y)) - (y+\delta_m - \xi_n(y))  K_{\jmath_n(y)}(y+\delta_m) \\
\stackrel{\mathclap{\eqref{eq:relations_boundary_crossing}}}{=} & &&  c_n(\xi_n(y+\delta_m)) - c_l(\xi_n(y+\delta_m)) + (y+\delta_m - \xi_n(y+\delta_m) ) K_l(y+\delta_m) \\
 &-&& c_n(\xi_n(y)) + c_{l}(\xi_n(y)) -(y-\xi_n(y)) (K_l(y) - K_{\jmath_n(y)}(y)) \\
 &-&& (y+\delta_m - \xi_n(y))  K_{\jmath_n(y)}(y+\delta_m).
\end{alignat*}
From this we obtain for almost all $y \geq 0$ by using the induction hypothesis 
\begin{alignat}{3}
 & && \lim_{m \to \infty} \frac{c^n(\xi_n(y+\delta_m),y+\delta_m) - c^n(\xi_n(y),y+\delta_m)}{\delta_m} \nonumber \\
=& && \xi'_n(y+) \Big[ c'_n(\xi_n(y)+) - c'_l(\xi_n(y)+) - K_l(y) \Big]  \nonumber \\
 &+&& K_l(y) + (y-\xi_n(y)) K'_l(y) - K_{\jmath_n(y)}(y) - (y-\xi_n(y)) K'_{\jmath_n(y)}(y) \nonumber \\
\stackrel[\mathclap{\eqref{eq:app_ODE_K_n}}]{\mathclap{\eqref{eq:optimality_xi}}}{=} & && -\xi'_n(y+) K_n(y).
\label{eq:computation_derivative}
\end{alignat}

Together with \eqref{eq:ODE_c_n} this yields in the case when $c^n(\cdot,y)$ is differentiable at $\xi_n(y)$
\begin{align}
K'_n(y) &= \frac{\xi'_n(y)-1}{y-\xi_n(y)} K_n(y) + \frac{1}{y-\xi_n(y)}\left( -K_n(y)\xi'_n(y) + \frac{\partial c^n}{\partial y}(\xi_n(y),y) \right) \nonumber \\
		&= -\frac{K_n(y)}{y-\xi_n(y)} + \frac{1}{y-\xi_n(y)} \Big( K_{\jmath_n(y)}(y) + (y-\xi_{n}(y))  K'_{\jmath_n(y)}(y) \Big).
		\label{eq:ODE_K_N_proof}
\end{align}
In order to finish the proof we just have to establish that \eqref{eq:ODE_K_N_proof} also holds in the case when $c^n(\cdot, y) $ is not differentiable at $\xi_n(y)$.

To this end, we first argue that \eqref{eq:computation_derivative}, and hence \eqref{eq:ODE_K_N_proof}, remains true in the case when $c^n(\cdot,y)$ is not differentiable at $\xi_n(y)$, but when the slope of the supporting tangent to $c^n(\cdot,y)$ at $\xi_n(y)$ which passes the $x$-axis at $y$ equals the right-derivative of $c^n(\cdot,y)$ at $\xi_n(y)$. In that case, denoting $k=\imath_n(\xi_n(y)+;y)$,
\begin{align}
c'_n(\xi_n(y)+) - c'_k(\xi_n(y)+) - K'_k(y) = -K_n(y).
\label{eq:optimality:right_continuous_version_1}
\end{align}
Recall the sequence $(\delta_m)$. We do not necessarily have $k=\jmath_n(y+\delta_m)=l$. Nevertheless, we argue that 
\begin{align}
c'_n(\xi_n(y)+) - c'_l(\xi_n(y)+) - K_l(y) = -K_n(y)
\label{eq:optimality:right_continuous_version_required}
\end{align}
holds.
We can safely assume that $\xi'_n(y+)>0$ (in the other case the conclusion of \eqref{eq:ODE_K_N_proof} remains true). Also it is enough to consider the case when $k>\jmath_n(y+\delta)$ for all $\delta>0$ sufficiently small (otherwise we may consider an alternative sequence $(\delta_m)$ where $k=\jmath_n(y+\delta_m)$ for all $m$ and \eqref{eq:optimality:right_continuous_version_required} would follow from  \eqref{eq:optimality:right_continuous_version_1}). 
Consequently, $\xi'_k(y+) \geq \xi'_n(y+)>0$. Then, since by induction hypothesis \eqref{eq:optimality_xi} holds true for $k$, we must have, for almost all $y$ that
\begin{align}
c'_k(\xi_n(y)+) - c'_j(\xi_n(y)+) - K_j(y) = -K_k(y)
\label{eq:optimality:right_continuous_version_2}
\end{align}
Then, combining these results we conclude by \eqref{eq:optimality:right_continuous_version_1} and \eqref{eq:optimality:right_continuous_version_2} that indeed \eqref{eq:optimality:right_continuous_version_required} holds.

%We assume that $\xi_n(y)$ has been already found.
Now we consider the case of non-smoothness of $c^n(\cdot,y)$ at $\xi_n(y)$ and where $y$ is such that the slope of the supporting tangent to $c^n(\cdot,y)$ at $\xi_n(y)$ which crosses the $x$-axis at $y$ does not equal the right-derivative of $c^n(\cdot,y)$ at $\xi_n(y)$.
We show that in this case we have for sufficiently small $\delta>0$,
\begin{align}
\xi_n(y)=\xi_n(y+\delta) \qquad \text{and hence} \qquad \xi'_n(y+)=0,
\label{eq:xi_constant_proof}
\end{align}
which implies that \eqref{eq:ODE_K_N_proof} holds as well.  

To achieve this we place a suitable tangent to $c^n(\cdot,y)$ at $\xi_n(y)$. 
Since, by assumption, $c^n(\cdot,y)$ has a kink at $\xi_n(y)$ we have some flexibility to do that.
Recalling the tangent interpretation of \eqref{eq:definition_xi_n_tangent_interpretation} we know by choice of $\xi_n(y)$ that we can place a supporting tangent to $c^n(\cdot,y)$ at $\xi_n(y)$ which passes through the $x$-axis at $y$. 
Alternatively, by choice of $y$, we can place a tangent to $c^n(\cdot,y)$ at $\xi_n(y)$ which crosses the $x$-axis at some $y+\delta>y$. 
This implies that 
\begin{align}
\argmin_{\zeta \leq y}\frac{c^n(\zeta,y)}{y+\delta-\zeta} = \xi_n(y).
\label{eq:non-unique_tangent}
\end{align}

Assume first that $\xi_n(y)<y$. Denote $k=\imath_n(\xi_n(y)+;y)$. For simplicity of the argument let us also assume that $\xi_n(y)=\xi_k(y) \neq \xi_i(y)$ for all $i \neq k,n$. Also denote $j=\jmath_n(y) = \imath_n(\xi_n(y);y)$.

Now we will use \eqref{eq:non-unique_tangent} to deduce \eqref{eq:xi_constant_proof}.
By continuity and monotonicity of $\xi_n$ we have for $\delta>0$ small enough that $\xi_n(y) \leq \xi_n(y+\delta)<\xi_n(y)+\epsilon<y$ for some $\epsilon = \epsilon(\delta)>0$. By taking $\delta$ small enough we can also assume that $k=\max_{\zeta \leq \xi_n(y)+\epsilon} \imath_n(\zeta;y+\delta)$. 
Then we have
\begin{alignat}{3}
\inf_{\zeta \leq y+\delta } \frac{c^n(\zeta,y+\delta)}{y+\delta-\zeta} 
\geq & && \inf_{\xi_n(y) \leq \zeta < \xi_n(y)+\epsilon} \frac{c^n(\zeta,y)}{y+\delta-\zeta}  %\nonumber \\
     %&+&& 
     +\inf_{\xi_n(y)\leq \zeta \leq \xi_n(y)+\epsilon} \frac{c^n(\zeta,y+\delta) - c^n(\zeta,y)}{y+\delta-\zeta}  \label{eq:proof_xi_constant}
%\geq & && \inf_{\xi_n(y) \leq \zeta < \xi_n(y)+\epsilon} \frac{c^n(\zeta,y)}{y+\delta-\zeta}   
%+ \inf_{\xi_n(y)\leq \zeta \leq \xi_n(y)+\epsilon} l(\zeta,y).
\end{alignat}
As for the first infimum in \eqref{eq:proof_xi_constant} we know from \eqref{eq:non-unique_tangent} that it is attained at $\zeta=\xi_n(y)$. 
Now we will show that the second infimum in \eqref{eq:proof_xi_constant} is also attained at $\zeta = \xi_n(y)$. To this end consider the following estimate for $\zeta > \xi_n(y)$,
\begin{alignat}{3}
	 & \quad && c^n(\zeta,y+\delta) - c^n(\zeta,y) \nonumber \\
=    &-&& c_{\imath_n(\zeta;y+\delta)}(\zeta) +  c_{\imath_n(\zeta;y)}(\zeta) -(y-\zeta) K_{\imath_n(\zeta;y)}(y) + (y+\delta-\zeta) K_{\imath_n(\zeta;y+\delta)}(y+\delta) \nonumber \\
\stackrel{\mathclap{\eqref{eq:definition_xi_n_tangent_interpretation}}}{\geq} &-&& (y-\zeta) K_{\imath_{\imath_n(\zeta;y)}(\zeta;y)}(y) + (y+\delta-\zeta) K_{\imath_n(\zeta;y+\delta)}(y+\delta) \nonumber \\
\geq &-&&(y-\zeta)K_j(y) + (y+\delta-\zeta)K_j(y+\delta) =:l(\zeta,y;\delta).
\label{eq:lower_bound_line-difference_c_n}
\end{alignat}
Since $l(\cdot,y;\delta)$ is non-decreasing and non-negative, we deduce that 
\begin{align*}
\argmin_{\xi_n(y) \leq \zeta < \xi_n(y)+\epsilon} \frac{l(\zeta,y;\delta)}{y+\delta-\zeta} = \xi_n(y).
\end{align*}

Finally, because at $\zeta = \xi_n(y)$ there is equality in \eqref{eq:lower_bound_line-difference_c_n}
%\begin{alignat*}{3}
%& &&c^n(\xi_n(y),y+\delta) - c^n(\xi_n(y),y) \\
%=&-&&(y-\xi_n(y))K_{\imath_n(\xi_n(y);y)}(y) + (y+\delta-\xi_n(y)) K_{\imath_n(\xi_n(y);y+\delta)}(y+\delta) \\
%=&-&&(y-\zeta)K_j(y) + (y+\delta-\zeta)K_j(y+\delta)		
%\end{alignat*}
we can conclude
\begin{align*}
\argmin_{\zeta \leq y+\delta}\frac{c^n(\zeta,y+\delta)}{y+\delta-\zeta} = \xi_n(y)
\end{align*}
as required.

In the case when $\xi_n(y)=y$ we obtain by \eqref{eq:right_continuity_imath} for $\delta>0$ sufficiently small that $\imath_n(y;y+\delta) = \imath_n(y;y)$ and hence
\begin{alignat*}{3}
\argmin_{\zeta < y+\delta}\frac{c^n(\zeta,y+\delta)}{y+\delta-\zeta} =& \quad &&
\argmin_{y \leq \zeta < y+\delta} \left( \underbrace{\frac{c_n(\zeta)-c_{\imath_n(\zeta;y+\delta)}(\zeta)}{y+\delta-\zeta}}_{\geq 0}+\underbrace{ K_{\imath_n(\zeta;y+\delta)}(y+\delta)}_{\geq K_{\imath_n(y;y+\delta)}(y+\delta)} \right) \\
\stackrel{\mathclap{\eqref{eq:extension_c_n}}}{=}& &&\xi_n(y)=y.
\end{alignat*}
The proof is complete.
\end{proof}

\bibliographystyle{plainnat}
%\bibliographystyle{apalike}
% or: plain,unsrt,alpha,abbrv,acm,apalike,ieeetr
%\bibliography{home/spoida/Desktop/literature}
\bibliography{literature_nolinks}

\end{document}